\documentclass{article}

\usepackage{amscd,amsmath,amssymb,amsthm,xcolor,mathrsfs,dsfont}
\newcommand{\eps}{\varepsilon}

\newcommand{\R}{\mathbb R}
\newcommand{\N}{\mathbb N}

\newcommand{\then}{\Longrightarrow}
\newcommand{\J}{{\cal J}}

\newcommand{\M}{{\cal M}}
\newcommand{\enne}{{\cal N}}

\newcommand\meas{{\rm meas}}

\DeclareMathOperator{\codim}{codim}
\DeclareMathOperator*{\esssup}{ess\; sup}

\newtheorem{corollary}{Corollary}[section]
\newtheorem{theorem}[corollary]{Theorem}
\newtheorem{lemma}[corollary]{Lemma}
\newtheorem{proposition}[corollary]{Proposition}

\theoremstyle{definition}
\newtheorem{definition}[corollary]{Definition}
\newtheorem{remark}[corollary]{Remark}

\numberwithin{equation}{section}
\begin{document}

\title{{\bf Multiple solutions for\\
coupled gradient--type quasilinear elliptic systems\\
with supercritical growth}
\footnote{The research that led to the present paper was partially supported 
by MIUR--PRIN project ``Qualitative and quantitative aspects of nonlinear PDEs'' 
(2017JPCAPN\underline{\ }005), 
{\sl Fondi di Ricerca di Ateneo} 2017/18 ``Problemi differenziali non lineari''.
Both the authors are members of the Research Group INdAM--GNAMPA.}}

\author{Anna Maria Candela and Caterina Sportelli\\
{\small Dipartimento di Matematica}\\
{\small Universit\`a degli Studi di Bari Aldo Moro} \\
{\small Via E. Orabona 4, 70125 Bari, Italy}\\
{\small \it annamaria.candela@uniba.it, caterina.sportelli@uniba.it}}
\date{}

\maketitle

\begin{abstract}
In this paper we consider the following coupled gradient--type quasilinear elliptic system
\begin{equation*}
\left\{
\begin{array}{ll}
- {\rm div} ( a(x, u, \nabla u) ) 
+ A_t (x, u, \nabla u) = G_u(x, u, v) &\hbox{ in $\Omega$,}\\ [10pt]
- {\rm div} ( b(x, v, \nabla v) ) 
+ B_t(x, v, \nabla v) = G_v\left(x, u, v\right) &\hbox{ in $\Omega$,}\\[10pt]
u = v = 0 &\hbox{ on $\partial\Omega$,}
\end{array}
\right.
\end{equation*}
where $\Omega$ is an open bounded domain in $\R^N$, $N\geq 2$. 
We suppose that some $\mathcal{C}^{1}$--Carathéodory functions $A, B:\Omega\times\R\times\R^N\rightarrow\R$ 
exist such that $a(x,t,\xi) = \nabla_{\xi} A(x,t,\xi)$, 
$A_t(x,t,\xi) = \frac{\partial A}{\partial t} (x,t,\xi)$, 
$b(x,t,\xi) = \nabla_{\xi} B(x,t,\xi)$, $B_t(x,t,\xi) =\frac{\partial B}{\partial t}(x,t,\xi)$,
and that $G_u(x, u, v)$, $G_v(x, u, v)$ are the partial derivatives
of a $\mathcal{C}^{1}$--Carathéodory nonlinearity $G:\Omega\times\R\times\R\rightarrow\R$.
 
Roughly speaking, we assume that $A(x,t,\xi)$ grows at least as $(1+|t|^{s_1p_1})|\xi|^{p_1}$, 
$p_1 > 1$, $s_1 \ge 0$,
while $B(x,t,\xi)$ grows as $(1+|t|^{s_2p_2})|\xi|^{p_2}$, 
$p_2 > 1$, $s_2 \ge 0$, and that
$G(x, u, v)$ can also have a supercritical growth related to $s_1$ and $s_2$.

Since the coefficients depend on the solution and its grandient themselves, 
the study of the interaction of two different norms in a suitable Banach space is needed.

In spite of these difficulties, a variational approach is used to show 
that the system admits a nontrivial weak bounded solution and, 
under hypotheses of symmetry, infinitely many ones.
\end{abstract}

\noindent
{\it \footnotesize 2020 Mathematics Subject Classification}. {\scriptsize Primary: 35J50;
Secondary: 35B33, 35J92, 47J30, 58E05}.\\
{\it \footnotesize Key words}. {\scriptsize Coupled gradient--type quasilinear elliptic system, $p$--Laplacian type operator, 
supercritical growth, weak Cerami--Palais--Smale condition, Ambrosetti--Rabinowitz condition, Mountain Pass theorem, 
critical Sobolev exponent, nontrivial weak bounded solution, pseudo--eigenvalue}.


\section{Introduction} 

The study of partial differential equations involving nonlinearities with critical
or supercritical growths, is a very complex matter and for many critical and supercritical problems
some basic issues are mostly unknown or undiscovered.
For example, let us consider the quasilinear elliptic problem
\begin{equation} \label{Pintro}
\left\{
\begin{array}{ll}
- \Delta_p u = \lambda |u|^{p-2} u + |u|^{q-2} u &\hbox{ in $\Omega$,}\\
u = 0 &\hbox{ on $\partial\Omega$,}
\end{array}
\right.
\end{equation}
where $\Omega$ is an open bounded domain in $\R^N, N\geq 3$, 
and $1<p<N$. In spite of the simple looking structure of the problem, 
if we ask $q$ to be critical or supercritical from the viewpoint of the Sobolev Embedding Theorem, 
namely $ q\ge p^*=\frac{Np}{N-p}$, some significant difficulties arise.
Among other problems, in general the lack of compactness which occurs, does not guarantee  
even the existence of solutions, which has been derived only in few cases and 
frequently under assumptions on the shape of the domain $\Omega$
(for the classical nonexistence result due to the Poho\v{z}aev's identity 
see \cite{Po}, or also \cite[Theorem III.1.3]{Str}). 

The existence of positive solutions of \eqref{Pintro} has been successfully addressed either by adding 
some lower order term to the critical nonlinearity (see \cite{BN}) 
or by considering domains which are not starshaped (see, e.g., \cite{BC,COR})
if $p=2$, $q=2^*$ and $\lambda=0$, while the existence of sign--changing solutions 
of \eqref{Pintro} have been obtained 
if $p=2$, $q=2^*$ but $\lambda \ne 0$ (see, e.g., \cite{AS,CFS}).
To our knowledge, all the results carried over so far, are built on the key assertion 
that the functional associated to the critical problem \eqref{Pintro} 
satisfies the Palais--Smale condition even if only in certain ranges of energy. 

On the other hand, taking $p\neq 2$, 
due to the hardship in handling a quasilinear operator,
very few results of existence have been derived so far, 
not even under assumptions of symmetry on the domain (we refer to \cite{MP} for a wider discussion).
We limit ourselves to point out that, as derived in \cite{MeWi}, 
a Poho\v{z}aev type nonexistence result is not yet available 
for sign--changing solutions of \eqref{Pintro}, 
as the unique continuation principle for the $p$--Laplacian is not known,
while it has been proved for nonnegative solutions (see \cite{GV}).
However, the existence of a positive solution in a domain with a sufficiently 
small hole has been shown for $\frac{2N}{N + 2}\leq p \leq2$, as well as an existence 
and multiplicity result has been proved under further assumptions of symmetry 
(see \cite{CLAPP, MP, MSS, MeWi, Per} and references therein).

In spite of the mentioned difficulties,  in recent years there has been a marked increase of research 
in critical and supercritical problems. The interest in these problems is related to their similarity
to some variational problems which arise in Geometry and Physics where the lack of compactness also occurs. 
In this sense, one of the best known challenges is the so called Yamabe's problem, 
but also some examples related to the existence of extremal functions for isoperimetric inequalities, 
Hardy--Littlewood--Sobolev inequalities and trace inequalities can be addressed 
(see, e.g., \cite{Jac, Lieb, Lions}).
However, in general, also in the ``simplest'' cases some problems are still open
and some classical variational tools, largely used in the subcritical case, 
do not work in the critical and supercritical ones.

Anyway, recently, quasilinear problems which generalize
\[
\left\{
\begin{array}{ll}
- {\rm div} ((1 + A(x) |u|^{sp}) |\nabla u|^{p-2}\nabla u) + s A(x) |u|^{sp-2} u |\nabla u|^p\
 =\ |u|^{\mu-2}u
  &\hbox{in $\Omega$,}\\
u\ = \ 0 & \hbox{on $\partial\Omega$,}
\end{array}
\right.
\]
have been studied and, by means of a suitable variational setting,
the existence of infinitely many weak bounded solutions is proved 
also if the nonlinear term has a supercritical growth such as   
$2 < 1 + p < p(s+1) < \mu < p^*(s+1)$, when $A \in L^\infty(\Omega)$ is such that
$A(x) \ge \alpha_0 >0$ for a.e. $x \in \Omega$ (see \cite{CPSSUP} and also, for other approaches,
\cite{ABO,LWW}).
One of the most remarkable feature of this work is that, 
unlike the results mentioned above, both an existence and a multiplicity result 
have been provided in the supercritical case  
for a more general problem without taking any symmetry assumption on 
the domain $\Omega$. 

Here, following the ideas introduced in \cite{CPSSUP}, we 
look for solutions of the family of coupled gradient--type quasilinear elliptic systems
\begin{equation}    \label{system}
\left\{
\begin{array}{ll}
- {\rm div} ( a(x, u, \nabla u) ) 
+ A_t (x, u, \nabla u) = G_u(x, u, v) &\hbox{ in $\Omega$,}\\ [10pt]
- {\rm div} ( b(x, v, \nabla v) ) 
+ B_t(x, v, \nabla v) = G_v\left(x, u, v\right) &\hbox{ in $\Omega$,}\\[10pt]
u = v = 0 &\hbox{ on $\partial\Omega$,}
\end{array}
\right.
\end{equation}
where $\Omega$ is an open bounded domain in $\R^N$, $N \ge 2$,
and $A, B :\Omega \times \R\times\R^N \to \R$ are given functions with partial derivatives
\begin{align}      \label{Ata}
&A_t (x, t, \xi) =\displaystyle\frac{\partial A}{\partial t}(x, t, \xi), 
\qquad a(x, t, \xi) = \left(\frac{\partial A}{\partial \xi_1}(x, t, \xi), \dots, 
\frac{\partial A}{\partial \xi_N}(x, t, \xi)\right),\\
\label{Btb}
&B_t (x, t, \xi)=\displaystyle\frac{\partial B}{\partial t}(x, t, \xi), 
\qquad b(x, t, \xi) = \left(\frac{\partial B}{\partial \xi_1}(x, t, \xi), \dots, 
\frac{\partial B}{\partial \xi_N}(x, t, \xi)\right),
\end{align}
for a.e. $x \in \Omega$, for all $(t,\xi) \in \R\times\R^N$.
Moreover, a nonlinear function $G:\Omega\times\R\times\R \to \R$ exists so that
\begin{equation}\label{GuGv}
G_u(x, u,v) =\frac{\partial G}{\partial u}(x, u, v),\quad G_v(x, u,v) 
=\frac{\partial G}{\partial v}(x, u, v)
\quad\mbox{ for a.e. } x\in\Omega, \mbox{ all } (u,v)\in\R^2.
\end{equation}

Roughly speaking, here we assume that $A(x,u,\nabla u)$ 
grows at least as $(1+|u|^{s_1p_1})|\nabla u|^{p_1}$, 
$p_1 > 1$, $s_1 \ge 0$,
while $B(x,v,\nabla v)$ grows at least as $(1+|v|^{s_2p_2})|\nabla v|^{p_2}$, 
$p_2 > 1$, $s_2 \ge 0$ (see Remark \ref{remmu1} and assumption $(h_7)$), and that
$G(x, u, v)$ can also have a supercritical growth depending on $s_1$ and $s_2$
(see hypothesis $(g_2)$).

While subcritical quasilinear systems have been handled through several techniques 
(see, e.g., \cite{AG, dF1, CMPP, CSS, CS}), as far as we know very few existence results 
have been determined for supercritical quasilinear elliptic systems 
(see, for example, \cite{CMP, GPZ} and references therein), 
even though no result occurs for supercritical systems with coefficients depending 
on the solution and its gradient themselves, as that one in \eqref{system}.
Moreover, as in \cite{CPSSUP}, even if a supercritical growth occurs,  
the domain $\Omega$ is only open and bounded as we consider homogeneous Dirichlet 
boundary condition and the solutions we are looking for, are weak.

Thus, following the same approach used in \cite{CSS} and \cite{CS}, 
but carefully adapting the ideas in \cite{CPSSUP} to our supercritical setting, 
we give some sufficient conditions for recognizing the variational structure of problem \eqref{system},
so that investigating solutions of \eqref{system} reduces to find critical points of functional
\begin{equation}      \label{functional}
\J(u,v) = \int_{\Omega} A(x,u,\nabla u) dx +\int_{\Omega} B(x,v,\nabla v) dx -\int_{\Omega} G(x,u,v) dx
\end{equation}
in the product Banach space $X=X_1\times X_2$, with $X_i=W_0^{1, p_i}(\Omega)\cap L^{\infty}(\Omega)$
if $i\in\{1, 2\}$.

Moreover, since in the Banach space $X$ our functional $\J$ does not satisfy the 
Palais--Smale condition, or one of its standard variants,  
we are not allowed to use directly existence and multiplicity results 
as the classical Ambrosetti--Rabinowitz theorems stated in \cite{AR} or in \cite{BBF}.
Hence, we have to submit a weaker definition of the
Cerami's variant of Palais--Smale condition, the so--called
weak Cerami--Palais--Smale condition (see Definition \ref{wCPS}).
We believe that the use of this definition, introduced in the pioneering paper \cite{CP2}
and employed in the framework of a quasilinear supercritical system, 
represents another major improvement of the work in this field.  
In fact, here Definition \ref{wCPS} is used for stating an 
extended Mountain Pass Theorem and also its symmetric version 
of which we avail to gain our existence and multiplicity results (see Theorems \ref{mountainpass} and \ref{abstract}), 
but we do not exclude the chance that this feature may be also employed 
to recover other kind of problems (see, e.g., \cite{MeSq}). 
In fact, we highlight that this technique has been adapted to address problems 
placed over unbounded domains both in radial and in non--radial setting 
(see \cite{CS2020}, respectively \cite{CSS2}) but so far only in subcritical growth assumptions
(in \cite{AMM} the existence of solutions for some critical and supercritical problems
have been proved by using a different (radial) approach, 
which is not applicable for non-autonomous equations).

On the other hand, this enhancement imposes to {\sl pay the price} 
that some technical assumptions on the involved functions are needed. 
Namely, if we just assume the Carathéodory functions $A(x, t, \xi)$, $B(x, t, \xi)$, $G(x, u, v)$ 
and their partial derivatives fit some proper polynomial growths to show 
the $\mathcal{C}^1$ regularity of the functional $\J$ in \eqref{functional},  
on the other hand the proof of the weak Cerami--Palais--Smale condition passes 
through some fine requirements on the involved functions (see Section \ref{sec_wcps})
and a very remarkable result (see Lemma \ref{Ladyz}) which has interest own self and 
can be employed regardless of this scenario to fix a problem of common trouble in this field.

Now, in order to draw the attention to the enhancement of our main results, 
we state them here in a ``streamlined'' version but we refer the reader to 
Section \ref{sec_main} for all the needed hypotheses on the involved functions
and the precise statement of the results (see Theorems \ref{ThExist} and \ref{ThMolt}).

\begin{theorem} \label{uno1}
Suppose that $A(x,t,\xi)$ grows at least as $(1+|t|^{s_1p_1})|\xi|^{p_1}$, 
with $p_1 > 1$, $s_1 \ge 0$, while $B(x,t,\xi)$ grows at least as $(1+|t|^{s_2p_2})|\xi|^{p_2}$, 
with $p_2 > 1$, $s_2 \ge 0$. Moreover, assume that
the $\mathcal{C}^1$--Carathéodory function $A(x,t,\xi)$, respectively $B(x,t,\xi)$, 
and its partial derivatives fits some suitable interaction properties among themselves,
while the $\mathcal{C}^1$--Carathéodory nonlinear term $G(x,u,v)$ 
satisfies the Ambrosetti--Rabinowitz condition for systems
with coefficients $\theta_1$, $\theta_2 > 0$ such that
$\theta_i <\frac{1}{p_i}$, $i \in \{1,2\}$, and has 
a proper polynomial growth which 
can also be supercritical depending on $s_1$ and $s_2$.
If
\[
\limsup_{(u,v)\to (0,0)} \frac{G(x, u, v)}{\vert u\vert^{p_1}+|v|^{p_2}}\ <\ 
\alpha_2\min\{\lambda_{1,1}, \lambda_{2,1}\}\quad
\hbox{uniformly a.e. in $\Omega$,}
\]
with $\lambda_{i,1}$ first eigenvalue of $-\Delta_{p_i}$ in $W_0^{p_i}(\Omega)$, $i\in\{1, 2\}$,
then problem \eqref{system} admits a nontrivial weak bounded solution.
\end{theorem}

\begin{theorem}
In the same hypotheses of Theorem \ref{uno1}, assume that
$A(x,\cdot,\cdot)$ and $B(x,\cdot,\cdot)$ are 
even in $\R\times\R^N$ while $G(x,\cdot, \cdot)$ is even in $\R^2$ for a.e. $x\in\Omega$. 
Then, if 
\[
\liminf_{|(u, v)|\to +\infty}\frac{G(x, u, v)}
{\vert u\vert^{\frac{1}{\theta_1}} +\vert v\vert^{\frac{1}{\theta_2}}}\ >\ 0
\quad \hbox{uniformly a.e. in $\Omega$,}
\]
problem \eqref{system} admits infinitely many distinct weak bounded solutions.
\end{theorem}

Finally, in order to better explain the required hypotheses, 
we consider the particular setting
\begin{equation} \label{ExAB}
A(x, t, \xi)= \frac{1}{p_1} (1+|t|^{s_1p_1}) |\xi|^{p_1},
\quad B(x, t, \xi) = \frac{1}{p_2} (1+|t|^{s_2p_2}) |\xi|^{p_2},
\end{equation}
and
\begin{equation}\label{ExG}
G(x, u, v) = \frac{1}{q_1}|u|^{q_1} + \frac{1}{q_2}|v|^{q_2} +c_*|u|^{\gamma_1} |v|^{\gamma_2},
\end{equation}
with $c_*\ge 0$ and some 
positive exponents $p_i$, $s_i$, $q_i$, $\gamma_i$ for each $i\in\{1, 2\}$.
So, $\J$ in \eqref{functional} reduces to the functional $\J_0 : X \to \R$ such that
\[
\begin{split}
\J_0(u, v)&= \frac{1}{p_1}\ \int_{\Omega} (1+|u|^{s_1p_1}) |\nabla u|^{p_1} dx +
\frac{1}{p_2}\ \int_{\Omega} (1+|v|^{s_2p_2}) |\nabla v|^{p_2} dx\\
&\quad -\int_{\Omega}\left(\frac{1}{q_1}|u|^{q_1} + \frac{1}{q_2}|v|^{q_2} 
+c_*|u|^{\gamma_1} |v|^{\gamma_2}\right) dx,
\end{split}
\]
and, in a suitable set of assumptions, system \eqref{system} turns into the model problem
\begin{equation}\label{j0}
\left\{
\begin{array}{ll}
- {\rm div} ( (1+|u|^{s_1p_1}) |\nabla u|^{p_1-2} \nabla u) 
+ s_1 |u|^{s_1p_1-2} u\ |\nabla u|^{p_1} &\\
\qquad\qquad\qquad = |u|^{q_1-2}u 
+\gamma_1c_*|u|^{\gamma_1-2} u|v|^{\gamma_2}&\hbox{ in $\Omega$,}\\ [10pt]
- {\rm div} ( (1+|v|^{s_2p_2}) |\nabla v|^{p_2-2} \nabla v) 
+ s_2 |v|^{s_2p_2-2} v\ |\nabla v|^{p_2}&\\
\qquad\qquad\qquad  = 
\gamma_2 c_* |u|^{\gamma_1} |v|^{\gamma_2-2} v + |v|^{q_2-2} v &\hbox{ in $\Omega$,}\\[10pt]
u = v = 0 &\hbox{ on $\partial\Omega$.}
\end{array}
\right.
\end{equation}
Hence, the previous results can be reworded in this way.

\begin{theorem}\label{ThModel}
Let $A(x,t,\xi)$, $B(x,t,\xi)$ and $G(x, u, v)$ be as in \eqref{ExAB} and \eqref{ExG}
with $p_i>1$, $i \in \{1,2\}$, $c_* \ge 0$ and either $p_1< N$ or $p_2< N$. Assume that $\theta_1, \theta_2$ exist such that
\begin{equation}\label{exj0}
2 < 1 + p_i < p_i(s_i+1) <\frac{1}{\theta_i} \le q_i <p_i^*(s_i+1) 
\quad \mbox{for } i\in\{1, 2\},
\end{equation}
where $p^*_1$, $p_2^*$ are the critical Sobolev exponents, and also
\begin{equation}\label{exj01}
1<\gamma_1<q_1,\quad 1<\gamma_2<q_2\quad
\hbox{are such that}\quad \gamma_1\theta_1+\gamma_1\theta_2\ge 1.
\end{equation}
Then, if 
\begin{equation}\label{exj02}
\gamma_j \frac{q_i-1}{q_i-\gamma_i} < \frac{p_i}{N} \left(1 - \frac{1}{p^*_i(s_i+1)}\right) p_j^*(s_j+1)
\quad \mbox{ for } i, j\in\{1, 2\}, i\neq j,
\end{equation}
problem \eqref{j0} admits infinitely many weak bounded distinct solutions.
\end{theorem}

Our paper is organized as follows.
In Section \ref{abstractsection} we introduce the abstract setting needed to recognize the variational structure 
of our problem \eqref{system}, as well as some extended versions of the Mountain Pass Theorems are shown up. 
Furthermore, a regularity result for the functional $\J$ in \eqref{functional} is provided, too. 
Then, in Section \ref{sec_wcps} some further assumptions on $A(x, t, \xi)$, $B(x, t, \xi)$
and $G(x, u, v)$ are addressed in order to show that the functional $\J$ verifies 
the weak Cerami--Palais--Smale condition. 
Lastly, in Section \ref{sec_main} our main results are stated and proved.  


\section{Abstract tools and variational setting}
\label{abstractsection}

We denote $\N = \{1, 2, \dots\}$ and, as long as we introduce our abstract setting, we employ the following notations:
\begin{itemize}
\item $(X, \|\cdot\|_X)$ is a Banach space with dual 
$(X',\|\cdot\|_{X'})$,
\item $(W,\|\cdot\|_W)$ is a Banach space such that
$X \hookrightarrow W$ continuously, i.e. $X \subset W$ and a constant $\sigma_0 > 0$ exists
such that
\[ 
\|y\|_W \ \le \ \sigma_0\ \|y\|_X\qquad \hbox{for all $y \in X$,}
\]
\item $J : {\cal D} \subset W \to \R$ and $J \in \mathcal{C}^1(X,\R)$ with $X \subset {\cal D}$.
\end{itemize}

In order to avoid any
ambiguity and simplify, when possible, the notation, 
from now on by $X$ we denote the space equipped with
its given norm $\|\cdot\|_X$ while, if the norm $\Vert\cdot\Vert_{W}$ is involved,
we write it explicitly.

Now, taking $\beta \in \R$, we say that a sequence
$(y_n)_n\subset X$ is a {\sl Cerami--Palais--Smale sequence at level $\beta$},
briefly {\sl $(CPS)_\beta$--sequence}, if
\[
\lim_{n \to +\infty}J(y_n) = \beta\quad\mbox{and}\quad 
\lim_{n \to +\infty}\|dJ\left(y_n\right)\|_{X'} (1 + \|y_n\|_X) = 0.
\]

As pointed out in \cite[Example 4.3]{CP2017}, a $(CPS)_{\beta}$ sequence can be 
constructed so that it is unbounded in $\|\cdot\|_X$ but converges with 
respect to $\|\cdot\|_W$. Thus, as in \cite{CPSSUP}, 
we introduce the following definition.

\begin{definition} \label{wCPS}
The functional $J$ satisfies the
{\slshape weak Cerami--Palais--Smale 
condition at level $\beta$} ($\beta \in \R$), 
briefly {\sl $(wCPS)_\beta$ condition}, if for every $(CPS)_\beta$--sequence $(y_n)_n$,
a point $y \in X$ exists, such that 
\begin{description}{}{}
\item[{\sl (i)}] $\displaystyle 
\lim_{n \to+\infty} \|y_n - y\|_W = 0\quad$ (up to subsequences),
\item[{\sl (ii)}] $J(y) = \beta$, $\; dJ(y) = 0$.
\end{description}
We say that $J$ satisfies $(wCPS)$ in $I$, $I$ real interval, if $J$ 
satisfies the $(wCPS)_{\beta}$ condition in $X$ at each level $\beta\in I$.
\end{definition}

Anyway, even if we deal with a weaker version of the 
Cerami's variant of the Palais--Smale condition, 
some classical abstract results can be extended so to fit
to our purposes. Actually, as in \cite[Lemma 2.2]{CPSSUP} (see also \cite[Lemma 2.3]{CP3})
a Deformation Lemma can be stated which provides the following 
extended version of the Mountain Pass Theorem given in \cite{AR} 
(see \cite[Theorem 2.3]{CPSSUP} for a detailed proof).

\begin{theorem}
\label{mountainpass}
Let $J\in \mathcal{C}^1(X,\R)$ be such that $J(0) = 0$
and the $(wCPS)$ condition holds in $\R_+$.\\
Moreover, assume that there exist a continuous map
$\ell : X \to \R$, some constants $r_0$, $\varrho_0 > 0$, and  
 $e \in X$ such that 
\begin{itemize}
\item[$(i)$] $\; \ell(0) =0 \qquad\hbox{and}\qquad \ell(y) \geq \|y\|_W 
\quad \hbox{for all $y \in X$}$;
\item[$(ii)$] $\; y \in X, \quad \ell(y) = r_0\qquad \then\qquad J(y) \geq \varrho_0$;
\item[$(iii)$] $\; \| e\|_W > r_0\qquad\hbox{and}\qquad J(e) < \varrho_0$.
\end{itemize}
Then, $J$ has a Mountain Pass critical point $y^* \in X$ such that $J(y^*) \geq \varrho_0$.
\end{theorem}

If, in addition, we require that $J$ is symmetric, 
then a more general version of the Symmetric Mountain Pass Theorem 
in \cite{AR}
can be stated, too (for the proof, see \cite[Theorem 2.4]{CPSSUP}).

\begin{theorem}
\label{abstract}
Let $J\in \mathcal{C}^1(X,\R)$ be an even functional such that $J(0) = 0$ and
the $(wCPS)$ condition holds in $\R_+$.
Moreover, assume that $\varrho > 0$ exists so that:
\begin{itemize}
\item[$({\cal H}_{\varrho})$]
three closed subsets $V_\varrho$, $Z_\varrho$ and $\M_\varrho$ of $X$ and a constant
$R_\varrho > 0$ exist which satisfy the following conditions:
\begin{itemize}
\item[$(i)$] $V_\varrho$ and $Z_\varrho$ are subspaces of $X$ such that
\[
V_\varrho + Z_\varrho = X,\qquad \codim Z_\varrho\ <\ \dim V_\varrho\ <\ +\infty;
\]
\item[$(ii)$] $\M_\varrho = \partial \enne$, where $\enne \subset X$ is a neighborhood of the origin
which is symmetric and bounded with respect to $\|\cdot\|_W$; 
\item[$(iii)$]  $\ y \in \M_\varrho \cap Z_\varrho\qquad \then\qquad J(y) \ge \varrho$;
\item[$(iv)$] $\ y \in V_\varrho, \quad \|y\|_X \ge R_\varrho \qquad \then\qquad J(y) \le 0$.
\end{itemize}
\end{itemize}
Then, if we put 
\[
\beta_\varrho\ =\ \inf_{\gamma \in \Gamma_\varrho} \sup_{y\in V_\varrho} J(\gamma(y)),
\]
with 
\[
\Gamma_\varrho\ =\ \{\gamma : X \to X:\
\gamma\ \hbox{odd homeomorphism,}\quad \gamma(y) = y \ 
\hbox{if $y \in V_\varrho$ with $\|y\|_X \ge R_\varrho$}\},
\]
the functional $J$ possesses at least a pair of symmetric critical points in $X$ 
with corresponding critical level $\beta_\varrho$ which belongs to $[\varrho,\varrho_1]$,
where $\varrho_1 \ge \displaystyle \sup_{y \in V_\varrho}J(y) > \varrho$.
\end{theorem}

\begin{remark}
Since the vector space $V_{\varrho}$ in Theorem \ref{abstract} has finite dimension, 
then condition $(\cal{H}_{\varrho})$$(iv)$ implies that $\displaystyle\sup_{y\in V_{\varrho}} \J(y) <+\infty$. 
Moreover, such hypothesis still holds if we replace $\Vert\cdot\Vert_{X}$ with $\Vert\cdot\Vert_{W}$.
\end{remark}

Finally, if we can apply Theorem \ref{abstract} infinitely many times, then 
the following multiplicity abstract result can be stated, too.

\begin{corollary}
\label{multiple}
Let $J \in \mathcal{C}^1(X,\R)$ be an even functional such that $J(0) = 0$,
the $(wCPS)$ condition holds in $\R_+$ and 
assumption $({\cal H}_{\varrho})$ holds for all $\varrho > 0$.\\
Then, the functional $J$ possesses a sequence of critical points $(y_n)_n \subset X$ such that
$J(y_n) \nearrow +\infty$ as $n \nearrow +\infty$.
\end{corollary}

Now, we proceed introducing the notations related to our specific issue.
If $\Omega \subset \R^N$ is an open bounded domain, $N\ge 2$,
we denote by:
\begin{itemize}
\item $L^q(\Omega)$ the Lebesgue space with
norm $|y|_q = \left(\int_\Omega|y|^q dx\right)^{1/q}$ if $1 \le q < +\infty$;
\item $L^\infty(\Omega)$ the space of Lebesgue--measurable 
and essentially bounded functions $y :\Omega \to \R$ with norm 
$\displaystyle |y|_{\infty} = \esssup_{\Omega} |y|$;
\item $W^{1,p}_0(\Omega)$ the Sobolev space equipped with the
norm $\|y\|_{W_0^{1, p}} = |\nabla y|_{p}$ if $1 \le p < +\infty$;
\item $\mbox{meas}(D)$ the usual Lebesgue measure of a measurable set $D$ in $\R^N$;
\item $|\cdot|$ the standard norm on any Euclidean space, as the dimension 
of the vector taken into account is clear and no ambiguity occurs.
\end{itemize}
Moreover, for any $m\in\N$, we say that $h:\Omega\times\R^m\to\R$ is a 
$\mathcal{C}^{k}$--Carath\'eodory function, $k\in\N\cup\lbrace 0\rbrace$, if
$\ h(\cdot, \nu)$ is measurable in $\Omega$ for all $\nu \in \R^m$ while
$h(x,\cdot)$ is $\mathcal{C}^k$ in $\R^m$ for a.e. $x \in \Omega$.
\medskip

Let $A,\ B :\Omega\times\R\times\R^N\to\R$ be such that the following conditions hold:
\begin{itemize}
\item[$(h_0)$]
$A(x,t,\xi)$ and $B(x,t,\xi)$ are $\mathcal{C}^1$--Carath\'eodory 
functions with partial derivatives as in \eqref{Ata}, respectively \eqref{Btb};
\item[$(h_1)$] two exponents $p_1>1$, $p_2 > 1$, and some positive functions 
$\Phi_i, \phi_i, \Psi_i, \psi_i \in \mathcal{C}^0(\R,\R)$, if $i \in \{0,1,2\}$,
exist such that
\begin{eqnarray}
\label{Aphi0}
\vert A(x, t, \xi)\vert &\leq&\Phi_0(t) +\phi_0(t)\vert\xi\vert^{p_1}  
\qquad\hbox{ a.e. in } \Omega, \mbox{ for all } (t,\xi)\in\R\times\R^N,\\
\vert A_t(x, t, \xi)\vert &\leq&\Phi_1(t) +\phi_1(t)\vert\xi\vert^{p_1}  
\qquad\hbox{ a.e. in } \Omega, \mbox{ for all } (t,\xi)\in\R\times\R^N,
\nonumber\\
\label{suA2}
\vert a(x, t, \xi)\vert &\leq&\Phi_2(t) +\phi_2(t)\vert\xi\vert^{p_1-1}  
\quad\hbox{ a.e. in } \Omega, \mbox{ for all } (t,\xi)\in\R\times\R^N,
\end{eqnarray}
and
\begin{eqnarray}
\label{Bpsi0}
\vert B(x, t, \xi)\vert &\leq&\Psi_0(t) +\psi_0(t)\vert\xi\vert^{p_2}  
\qquad\mbox{ a.e. in } \Omega, \mbox{ for all } (t,\xi)\in\R\times\R^N,\\
\vert B_t(x, t, \xi)\vert &\leq&\Psi_1(t) +\psi_1(t)\vert\xi\vert^{p_2} 
 \qquad\mbox{ a.e. in } \Omega, \mbox{ for all } (t,\xi)\in\R\times\R^N, \nonumber\\
\label{suB2}
\vert b(x, t, \xi)\vert&\leq& \Psi_2(t) +\psi_2(t)\vert\xi\vert^{p_2-1} 
 \quad\mbox{ a.e. in } \Omega, \mbox{ for all } (t,\xi)\in\R\times\R^N.
\end{eqnarray}
\end{itemize}

Furthermore, let $G :\Omega\times\R\times\R\to\R$ be a map which satisfies 
the following hypotheses: 
\begin{itemize}
\item[$(g_0)$] 
$G(x,u,v)$ is a $\mathcal{C}^1$--Caratheodory function with partial derivatives as in \eqref{GuGv},
such that
\[
G(\cdot, 0, 0)\in L^{\infty}(\Omega)
\qquad \hbox{and}\qquad
G_u(x, 0, 0) = G_v(x, 0, 0) = 0 \quad \mbox{ for a.e. } x\in\Omega;
\]
\item[$(g_1)$] a constant $\sigma>0$ and some exponents $q_i \geq 1$, $t_i \geq 0$,
if $i\in\{1, 2\}$, exist such that
\begin{align}      \label{gucr}
\vert G_u(x,u,v)\vert&\leq \sigma(1 +\vert u\vert^{q_1 -1} +\vert v\vert^{t_1}) 
\quad\mbox{ for a.e. $x\in\Omega$, for all }(u, v)\in\R^2,\\     
\nonumber
\vert G_v(x,u,v)\vert&\leq \sigma(1 +\vert u\vert^{t_2} +\vert v\vert^{q_2 -1}) 
\quad\mbox{ for a.e. $x\in\Omega$, for all }(u, v)\in\R^2.
\end{align}
\end{itemize}

\begin{remark}\label{sucrit}
Hypotheses $(g_0)$--$(g_1)$, the Mean Value Theorem and direct computations ensure the existence 
of a positive constant $\sigma_1 > 0$ such that
\begin{equation}    \label{Gsigma}
\vert G(x,u,v)\vert\leq \sigma_1\left(1+\vert u\vert^{q_1} +\vert v\vert^{t_1}\vert u\vert 
+ \vert u\vert^{t_2}\vert v\vert + \vert v\vert^{q_2}\right)
\quad\mbox{ for a.e. $x\in\Omega$, for all }(u,v)\in\R^2.
\end{equation}
Now, taking any couple of real numbers $t_3$, $t_5 >1$, 
from Young inequality we obtain
\begin{equation} \label{crits12}
|v|^{t_1} |u| \ \le \ |u|^{t_3} + |v|^{t_4}, \qquad 
|u|^{t_2}|v| \ \le \ |u|^{t_6} + |v|^{t_5} \quad \hbox{for all $(u,v) \in \R^2$},
\end{equation}
where, for simplicity, we set 
\begin{equation} \label{crits120}
t_4 :=\ \frac{t_1 t_3}{t_3-1} \ge t_1\quad 
\hbox{and}\quad t_6 :=\ \frac{t_2 t_5}{t_5-1} \ge t_2.
\end{equation}
Thus, from \eqref{Gsigma} and \eqref{crits12} we infer that
\begin{equation}    \label{Gsigmamax}
| G(x,u,v)| \leq \sigma_2 (1 +\vert u\vert^{\overline{q}_1} +\vert v\vert^{\overline{q}_2})
\quad \hbox{for a.e. $x \in \Omega$, for all $(u,v) \in \R^2$},
\end{equation}
with 
\begin{equation}    \label{Ggrow}
\overline{q}_1 :=\max\lbrace q_1, t_3, t_6\rbrace\qquad
\hbox{and}\qquad
\overline{q}_2 := \max\lbrace q_2, t_4, t_5\rbrace, 
\end{equation}
for a suitable constant $\sigma_2 > 0$.
\end{remark}

In order to recall some features shared by the subcritical systems in \cite{CSS}
and \cite{CS}, if needed, here we introduce similar notations.

For each $ i\in\lbrace 1, 2\rbrace$
let $p_i > 1$ be as in assumption $(h_1)$ and 
let us consider the related Sobolev space
\[
W_i = W_0^{1, p_i}(\Omega)\quad \hbox{with norm $\Vert \cdot\Vert_{W_i} = \Vert \cdot\Vert_{W_0^{1, p_i}}$.}
\]
From the Sobolev Embedding Theorem, for any $r\in[1, p_i^{\ast}]$ 
with $p_i^{\ast} = \frac{Np_i}{N-p_i}$ if $N>p_i$, 
or any $r\in[1, +\infty[$ if $p_i \ge N$, $W_i$ is
continuously embedded in $L^{r}(\Omega)$, i.e., $\tau_{i,r} > 0$ exists such that
\begin{equation}   \label{Sobpi}
\vert y\vert_r \leq\tau_{i,r}\Vert y\Vert_{W_i} \quad \mbox{ for all } y\in W_i.
\end{equation}
Furthermore, if $p_i\geq N$, we place
\[
p^*_i = +\infty\quad \hbox{and} \quad \frac{1}{p^*_i} = 0.
\]

Here, the notation $(W, \|\cdot\|_W)$, introduced for the abstract
setting at the beginning of this section, is referred to our problem with 
\begin{equation}     \label{Wdefn1}
W =W_1\times W_2\qquad \hbox{and}\qquad
\Vert (u, v)\Vert_W = \| u\|_{W_1} +\| v\|_{W_2}\quad
\hbox{if $(u,v) \in W$.}
\end{equation}
Since $(W_i, \Vert\cdot\Vert_{W_i})$ is a reflexive Banach space for both $i\in\{1, 2\}$,
so is $\left(W, \Vert\cdot\Vert_W\right)$ in \eqref{Wdefn1}.

Moreover, we consider the Banach space $(X,\|\cdot\|_X)$ defined as 
\begin{equation}     \label{Xdefn1}
X =X_1\times X_2 \qquad \hbox{with}\qquad
\Vert (u, v)\Vert_X = \| u\|_{X_1} +\| v\|_{X_2}\quad
\hbox{if $(u,v) \in X$,}
\end{equation}
where
\begin{equation}    \label{Xidefn}
X_1:= W_1\cap L^{\infty}(\Omega)\quad\mbox{ and }\quad X_2:= W_2\cap L^{\infty}(\Omega)
\end{equation}
are endowed with the norms
\[
\Vert u\Vert_{X_1} = \Vert u\Vert_{W_1} + \vert u\vert_{\infty}\; \mbox{ if } u\in X_1
\quad\mbox{ and }\quad\Vert v\Vert_{X_2} = \Vert v\Vert_{W_2} + \vert v\vert_{\infty}\; 
\mbox{ if } v\in X_2.
\]
Setting 
\[
L:=L^{\infty}(\Omega)\times L^{\infty}(\Omega)\quad \hbox{with}\quad 
\Vert (u, v)\Vert_{L} =|u|_{\infty} +|v|_{\infty},
\]
we have that $X$ in \eqref{Xdefn1} can also be written as 
\begin{equation}     \label{Xdefn}
X = W\cap L 
\end{equation}
and its norm is such that
\[   
\Vert (u, v)\Vert_X = \Vert (u, v)\Vert_W + \Vert (u, v)\Vert_L .
\]

Clearly, from \eqref{Xidefn}, for both $i\in\{1, 2\}$ we have that the continuous 
embeddings $X_i \hookrightarrow W_i$ and $X_i \hookrightarrow L^\infty(\Omega)$
hold. 

\begin{remark}
If $p_i > N$ for both $i\in\{1, 2\}$, then the embedding 
$W_i \hookrightarrow L^\infty(\Omega)$ means that $X_i=W_i$. Thus, $X = W$ and
the classical Mountain Pass Theorems in \cite{AR} may be applied.
\end{remark}

Firstly, we note that if conditions $(h_0)$--$(h_1)$, $(g_0)$--$(g_1)$ hold,
then direct computations imply that $\J(u,v)$ in \eqref{functional}
is well defined for all $(u, v)\in X$. Moreover,
taking any $(u, v)$, $(w, z)\in X$, the G\^ateaux differential 
of functional $\J$ in $(u,v)$ along the direction $(w,z)$ is given by
\begin{equation}     \label{diff}
\begin{split}
d\J(u, v)[(w, z)] = &\int_{\Omega} a(x, u, \nabla u)\cdot\nabla w \ dx 
+ \int_{\Omega} A_u(x, u, \nabla u) w \ dx + \int_{\Omega} b(x, v, \nabla v)\cdot \nabla z \ dx\\
&+ \int_{\Omega} B_v(x, v, \nabla v) z \ dx-\int_{\Omega} G_u(x, u, v) w \ dx - \int_{\Omega} G_v(x, u, v) z \ dx.
\end{split}
\end{equation}
For simplicity, we set
\[
\begin{split}
&\frac{\partial\J}{\partial u}(u, v): 
w\in X_1\mapsto \frac{\partial\J}{\partial u}(u, v)[w] = d\J(u, v)[(w, 0)]\in\R, \\
&\frac{\partial\J}{\partial v}(u, v): 
z\in X_2\mapsto \frac{\partial\J}{\partial v}(u, v)[z] = d\J(u, v)[(0, z)]\in\R;
\end{split}
\]
hence, from \eqref{diff}, it follows that
\begin{equation}     \label{dJw}
\begin{split}
\frac{\partial\J}{\partial u}(u, v) [w] = &\int_{\Omega} a(x, u, \nabla u)\cdot\nabla w\ dx 
+ \int_{\Omega} A_u(x, u, \nabla u) w \ dx - \int_{\Omega} G_u(x, u, v) w\ dx
\end{split}
\end{equation}
and
\begin{equation}    \label{dJz}
\begin{split}
\frac{\partial\J}{\partial v}(u, v) [z] = 
&\int_{\Omega} b(x, v, \nabla v)\cdot\nabla z\ dx + 
\int_{\Omega} B_v(x, v, \nabla v) z \ dx - \int_{\Omega} G_v(x, u, v) z \ dx.
\end{split}
\end{equation}
Taking $(u, v)\in X$, since $d\J(u, v)\in X^{\prime}$, 
then 
\[
\frac{\partial\J}{\partial u}(u, v)\in X_1^{\prime}, \qquad
\frac{\partial\J}{\partial v}(u, v)\in X_2^{\prime}
\]
and
\begin{equation}    \label{dJz1}
d\J(u, v)[(w, z)] = \frac{\partial\J}{\partial u}(u, v)[ w] 
+ \frac{\partial\J}{\partial v}(u, v)[z] \quad \mbox{ for all } (w, z)\in X.
\end{equation}
Furthermore, above remarks and direct computations give not only the estimates
\begin{equation}    \label{dJzstar}
\left\|\frac{\partial\J}{\partial u}(u, v)\right\|_{X_1'} \le \|d\J(u, v)\|_{X'}\quad\mbox{ and }\quad
\left\|\frac{\partial\J}{\partial v}(u, v)\right\|_{X_2'} \le \|d\J(u, v)\|_{X'},  
\end{equation}
but also
\[ 
\|d\J(u, v)\|_{X'} \le \left\|\frac{\partial\J}{\partial u}(u, v)\right\|_{X_1'} +
\left\|\frac{\partial\J}{\partial v}(u, v)\right\|_{X_2'}.
\]
At last, from \eqref{dJz1} we infer that
\[
d\J(u, v) = 0 \;\hbox{in $X$}\quad\iff\quad
\frac{\partial\J}{\partial u}(u, v)= 0 \;\hbox{in $X_1$}
\;\hbox{ and }\; \frac{\partial\J}{\partial v}(u,v) = 0 \;\hbox{in $X_2$.} 
\]

Finally, we can state the regularity of functional $\J$ defined in \eqref{functional}
(for the proof, see \cite[Proposition 3.5]{CS}).

\begin{proposition}\label{smooth1}
Assume that conditions $(h_0)$--$(h_1)$, $(g_0)$--$(g_1)$ hold. 
Let $((u_n, v_n))_n \subset X$ and $(u, v) \in X$ be such that
\[
(u_n, v_n)\to (u, v) \mbox{ in $W\quad $ and}  \quad (u_n, v_n)\to (u, v) 
\mbox{ a.e. in } \Omega\quad \mbox{ if } n\to+\infty. 
\]
If $M > 0$ exists such that
\[
\vert u_n\vert_{\infty}\leq M \quad \mbox{ and } \quad \vert v_n\vert_{\infty}\leq M \quad 
\hbox{for all $n\in\N$,}
\]
then
\[
\J(u_n, v_n) \to \J(u, v) \quad \mbox{ and } \quad 
\Vert d\J(u_n, v_n) - d\J(u, v)\Vert_{X^{\prime}}\to 0 \ \mbox{ as } n\to+\infty.
\]
Hence, $\J$ is a $\mathcal{C}^1$ functional on $X$ with Fr\'echet differential  
defined as in (\ref{diff}).
\end{proposition}


\section{The set up for the weak Cerami--Palais--Smale condition}
\label{sec_wcps}

In order to prove some more properties of functional $\J$ in \eqref{functional}, 
let $p_1>1$ and $p_2>1$ as in the earlier hypothesis $(h_1)$. Then, assume that 
$R \geq 1$ exists such that the following conditions hold:
\begin{itemize}
\item[$(h_2)$] some constants $\eta_1, \eta_2> 0$ exist such that
\begin{eqnarray}   
&&A(x, t, \xi)\leq\eta_1 a(x, t, \xi)\cdot\xi \quad \mbox{a.e. in $\Omega$ if }\ \vert(t, \xi)\vert\geq R,
\label{eta1}\\
&&B(x, t, \xi)\leq\eta_1 b(x, t, \xi)\cdot\xi \quad \mbox{a.e. in $\Omega$ if }\ \vert(t, \xi)\vert\geq R,
\label{eta12}
\end{eqnarray}
and
\begin{equation}   \label{2sup}
\sup_{\vert (t, \xi)\vert\leq R} \vert A(x, t, \xi)\vert\leq\eta_2, 
\qquad \sup_{\vert (t, \xi)\vert\leq R}\vert B(x, t, \xi)\vert\leq\eta_2 \quad \mbox{ a.e. in } \Omega;
\end{equation}
\item[$(h_3)$] some exponents $s_1, s_2\geq 0$ and a constant $\mu_0 > 0$ exist so that
\[
\begin{split}
&a(x, t, \xi)\cdot\xi\geq \mu_0 (1 +\vert t\vert^{s_1p_1})\vert\xi\vert^{p_1} 
\quad \mbox{ a.e. in } \Omega, \mbox{ for all } (t, \xi)\in\R\times\R^N,\\
&b(x, t, \xi)\cdot\xi\geq \mu_0 (1+\vert t\vert^{s_2p_2})\vert\xi\vert^{p_2} 
\quad \mbox{ a.e. in } \Omega, \mbox{ for all } (t, \xi)\in\R\times\R^N;
\end{split}
\]
\item[$(h_4)$] a constant $\mu_1 >0$ exists such that
\[
\begin{split}
& a(x, t, \xi)\cdot\xi + A_t(x, t, \xi) t \geq \mu_1 a(x, t, \xi)\cdot\xi 
\quad \mbox{ a.e. in } \Omega \mbox{ if } \vert(t, \xi)\vert\geq R,\\
& b(x, t, \xi)\cdot\xi + B_t(x, t, \xi) t \geq \mu_1 b(x, t, \xi)\cdot\xi 
\;\quad \mbox{ a.e. in } \Omega \mbox{ if } \vert(t, \xi)\vert\geq R; 
\end{split}
\]
\item[$(h_5)$] some constants $\theta_1, \theta_2, \mu_2 >0$ exist such that
\begin{equation}   \label{thi<pi}
\theta_1 <\frac{1}{p_1},\qquad\theta_2 <\frac{1}{p_2},
\end{equation}
and
\[
\begin{split}
A(x, t, \xi) - \theta_1 a(x, t,\xi)\cdot\xi -\theta_1 A_t(x, t, \xi) t &\geq\mu_2 a(x, t, \xi)\cdot\xi 
\quad \mbox{ a.e. in } \Omega \mbox{ if } \vert(t, \xi)\vert\geq R,\\
B(x, t, \xi) - \theta_2 b(x, t,\xi)\cdot\xi - \theta_2 B_t(x, t, \xi) t &\geq\mu_2  b(x, t, \xi)\cdot\xi 
\,\quad \mbox{ a.e. in } \Omega \mbox{ if } \vert(t, \xi)\vert\geq R;\\
\end{split}
\]
\item[$(h_6)$] for all $\xi, \xi^{\prime}\in\R^N$, with $\xi\neq\xi^{\prime}$, it is 
\[
\begin{split}
[a(x, t, \xi) -a(x, t, \xi^{\prime})] &\cdot [\xi-\xi^{\prime}] >0 
\quad \mbox{ a.e. in } \Omega, \ \mbox{ for all } t\in\R,\\
[b(x, t, \xi) -b(x, t, \xi^{\prime})] &\cdot [\xi-\xi^{\prime}] >0 
\quad \mbox{ a.e. in } \Omega, \ \mbox{ for all } t\in\R;
\end{split}
\]
\item[$(g_2)$] for $i\in\{1, 2\}$, taking $p_i$ as in hypothesis $(h_1)$, $q_i, t_i$ 
as in assumption $(g_1)$ and $s_i$ as in $(h_3)$, we assume that
\begin{equation}    \label{crit_exp}
1 \le q_1 < p_1^{\ast}(s_1 +1),\qquad 1 \le q_2 < p_2^{\ast}(s_2+1),
\end{equation}
and
\begin{equation}    \label{crit_expi}
0 \le t_1 < \frac{p_1}{N} \left(1 - \frac{1}{p^*_1(s_1+1)}\right) p_2^*(s_2+1), \quad
0 \le t_2 < \frac{p_2}{N} \left(1 - \frac{1}{p^*_2(s_2+1)}\right) p_1^*(s_1+1);
\end{equation}
\item[$(g_3)$] taking $\theta_1, \theta_2$ as in $(h_5)$, we assume that
\[
0 < G(x, u, v)\leq\theta_1 G_u(x, u, v) u +\theta_2 G_v(x, u, v) v 
\quad \mbox{ a.e. in } \Omega, \ \mbox{ if } \vert (u, v)\vert\geq R.
\]
\end{itemize}

\begin{remark}  \label{spiega}
Assumption \eqref{crit_exp} shows up the supercritical nature of our problem 
which vanishes if $s_1 = s_2 =0$ as it reduces exactly to the subcritical condition $(g_1)$ in \cite{CSS,CS}. 
However, in general, if one or both $s_1 >0$, $s_2>0$ hold, then
a supercritical growth on the nonlinear term $G(x, u, v)$ is allowed.
Moreover, we emphasize that the growth hypothesis $(g_2)$ is needed to prove 
that the functional $\J$ satisfies the $(wCPS)$ condition, but has not been required 
for the variational principle stated in Proposition \ref{smooth1}. 
\end{remark}

\begin{remark}    \label{remmu1}
If we consider hypothesis $(h_4)$ with $t = 0$ and $\vert\xi\vert\geq R$,
then assumption $(h_3)$ gives $\mu_1\leq 1$. Moreover, we note that $(h_4)$ and $(h_5)$ yield
\begin{align}    \label{Anew}
&A(x, t, \xi)\geq(\theta_1\mu_1 +\mu_2)\ a(x, t, \xi)\cdot\xi 
\quad \mbox{ a.e. in } \Omega \ \mbox{ if } \vert(t, \xi)\vert\geq R,\\    \label{Bnew}
&B(x, t, \xi)\geq(\theta_2\mu_1 +\mu_2)\ b(x, t, \xi)\cdot\xi 
\quad \mbox{ a.e. in } \Omega \ \mbox{ if } \vert(t, \xi)\vert\geq R
\end{align}
which, together with condition $(h_3)$, imply that
\begin{align}   \label{A13}
&A(x, t, \xi)\geq \mu_0 (\theta_1 \mu_1 +\mu_2)(1 +\vert t\vert^{s_1p_1})\vert\xi\vert^{p_1}\geq 0 
\quad \mbox{ a.e. in } \Omega \ \mbox{ if } \vert(t, \xi)\vert\geq R,\\    \label{B24}
&B(x, t, \xi)\geq \mu_0 (\theta_2 \mu_1+\mu_2)(1 +\vert t\vert^{s_2p_2})\vert\xi\vert^{p_2}\geq 0 
\quad \mbox{ a.e. in } \Omega \ \mbox{ if } \vert(t, \xi)\vert\geq R.
\end{align}
Hence, from \eqref{2sup} and \eqref{A13}, respectively \eqref{B24}, and direct computations 
we have that
\begin{align} \label{A13th}
&A(x, t, \xi)\geq \mu_0 (\theta_1 \mu_1 +\mu_2)(1 +\vert t\vert^{s_1p_1})\vert\xi\vert^{p_1}-\eta_3 
\quad \mbox{ a.e. in } \Omega \ \mbox{ for all } \ (t, \xi)\in\R\times\R^N,\\ \label{B24th}
&B(x, t, \xi)\geq \mu_0 (\theta_2 \mu_1 +\mu_2)(1 +\vert t\vert^{s_2p_2})\vert\xi\vert^{p_2}-\eta_3 
\quad \mbox{ a.e. in } \Omega \ \mbox{ for all } \ (t, \xi)\in\R\times\R^N,
\end{align}
for a suitable constant $\eta_3>0$. 
On the other hand, from $(h_2)$ and \eqref{suA2}, respectively \eqref{suB2}, 
direct computations imply that
\begin{align}   \label{perh0A}
A(x,t,\xi)&\leq \eta_1\Phi_2(t) + \eta_1 (\Phi_2(t) +\phi_2(t))\vert\xi\vert^{p_1} +\eta_4 
\,\quad \mbox{ a.e. in } \Omega \ \mbox{ for all } (t, \xi)\in\R\times\R^N,\\    \label{perh0B}
B(x, t, \xi)&\leq \eta_1\Psi_2(t) + \eta_1 (\Psi_2(t) +\psi_2(t)) \vert\xi\vert^{p_2} + \eta_4 
\quad \mbox{ a.e. in } \Omega \ \mbox{ for all } (t, \xi)\in\R\times\R^N,
\end{align}
for a suitable constant $\eta_4>0$. 
Then, if hypotheses $(h_2)$--$(h_5)$ hold, 
the growth conditions on $A(x, t, \xi)$ and $B(x, t, \xi)$ stated in \eqref{Aphi0} and \eqref{Bpsi0}
are a direct consequence of \eqref{suA2}, respectively \eqref{suB2}, as 
\eqref{Aphi0} follows from \eqref{A13th} and \eqref{perh0A} while \eqref{Bpsi0}
follows from \eqref{B24th} and \eqref{perh0B}. Hence, even if \eqref{Aphi0} and \eqref{Bpsi0} are 
part of assumption $(h_1)$, they can be ruled--out from the hypotheses if $(h_2)$--$(h_5)$ hold, too.
\end{remark}

\begin{remark}
In the set of hypotheses $(h_2)$ and $(h_5)$ a more precise growth condition
on both the functions $A(x, t, \xi)$ and $B(x, t, \xi)$ can be pointed out. 
In fact, \eqref{eta1}, respectively \eqref{eta12}, $(h_5)$ and direct calculations imply that
\begin{align}     \label{Alast}
&\left(\frac{\eta_1 - \theta_1 -\mu_2}{\eta_1\theta_1}\right) A(x, t, \xi)\geq A_t(x, t, \xi) t  
\quad \mbox{ a.e. in } \Omega \ \mbox{ if } \vert(t, \xi)\vert\geq R\\    \label{Blast}
&\left(\frac{\eta_1 - \theta_2 -\mu_2}{\eta_1\theta_2}\right) B(x, t, \xi)\geq B_t(x, t, \xi) t  
\quad \mbox{ a.e. in } \Omega \ \mbox{ if } \vert(t, \xi)\vert\geq R.
\end{align}
Now, taking $t=0$ and $|\xi|\geq R$ in both $(h_2)$ and $(h_5)$, 
without loss of generality we can choose $\mu_2$ small enough so that
\[
\eta_1 - \theta_1 -\mu_2>0\quad\mbox{ and }\quad \eta_1 - \theta_2 -\mu_2 >0.
\]
Thus, from \eqref{Aphi0}, \eqref{A13}, \eqref{Alast}, respectively
\eqref{Bpsi0}, \eqref{B24}, \eqref{Blast}, and direct computations we obtain that
\begin{align*} 
&A(x, t, \xi)\leq \eta_5 \vert t\vert^{\frac{\eta_1 -\theta_1 -\mu_2}{\eta_1\theta_1}}\vert\xi\vert^{p_1} 
\quad \mbox{ a.e. in } \Omega \mbox{ if } \vert t\vert\geq 1 \mbox{ and } \vert\xi\vert\geq R,\\ 
&B(x,t,\xi) \leq \eta_5 \vert t\vert^{\frac{\eta_1 - \theta_2 -\mu_2}{\eta_1\theta_2}}\vert\xi\vert^{p_2} 
\quad \mbox{ a.e. in } \Omega \mbox{ if } \vert t\vert\geq 1 \mbox{ and } \vert\xi\vert\geq R,
\end{align*}
for a suitable $\eta_5>0$, and then from \eqref{Anew}, respectively \eqref{Bnew},
we have that
\begin{align} \label{a_eta}
&a(x, t, \xi)\cdot\xi\leq\frac{\eta_5}{\theta_1\mu_1+\mu_2}
\vert t\vert^{\frac{\eta_1-\theta_1-\mu_2}{\eta_1\theta_1}}\vert\xi\vert^{p_1} 
\quad \mbox{ a.e. in } \Omega \mbox{ if } \vert t\vert\geq 1 
\mbox{ and } \vert\xi\vert\geq R,\\  \label{b_eta}
&b(x, t, \xi)\cdot\xi\leq\frac{\eta_5}{\theta_2\mu_1+\mu_2}
\vert t\vert^{\frac{\eta_1-\theta_2-\mu_2}{\eta_1\theta_2}}\vert\xi\vert^{p_2} 
\quad \mbox{ a.e. in } \Omega \mbox{ if } \vert t\vert\geq 1 \mbox{ and } \vert\xi\vert\geq R.
\end{align}
Finally, from \eqref{a_eta}, \eqref{b_eta} and assumption $(h_3)$, we infer that
\begin{equation}   \label{stimamin}
0\leq p_1 s_1 \le \frac{1}{\theta_1} -\frac{\theta_1 + \mu_2}{\eta_1\theta_1}
\quad\mbox{ and }\quad0\leq p_2 s_2 \le \frac{1}{\theta_2} -\frac{\theta_2+\mu_2}{\eta_1\theta_2}.
\end{equation}
We note that, if
\begin{equation} \label{si_pi}
0\leq s_1 <\frac{1}{\theta_1 p_1}\quad\mbox{ and }\quad 0\leq s_2 <\frac{1}{\theta_2 p_2},
\end{equation}
then we can always choose $\eta_1$ in $(h_2)$ large enough 
so that \eqref{stimamin} is satisfied.
\end{remark}

\begin{remark}
Conditions in \eqref{si_pi} not only relate the exponents $s_1, s_2$ 
provided in assumption $(h_3)$ with the powers $p_1, p_2 >1$ used in the growth conditions $(h_1)$ and 
$\theta_1, \theta_2$ claimed in \eqref{thi<pi}, but also they tell us how far we can take it. 
In particular, it implies that in our set of hypotheses, a supercritical growth is allowed as long as $s_1, s_2$ 
cover the whole range stated in \eqref{si_pi}.
\end{remark}

\begin{remark}  \label{rmkPerera0}
Assumptions $(g_0)$--$(g_1)$, $(g_3)$ and direct calculations imply that for each $i\in\{1, 2\}$ a function
$h_i\in L^{\infty}(\Omega)$, $h_i(x) >0$ for a.e. $x\in\Omega$,
exists such that
\begin{align*}
G(x,u,0)& \geq h_1(x) |u|^{\frac{1}{\theta_1}}\quad \mbox{ for a.e. } x\in\Omega, \mbox{ if } |u| \ge R,\\
G(x,0,v)& \geq h_2(x) |v|^{\frac{1}{\theta_2}}\quad \mbox{ for a.e. } x\in\Omega, \mbox{ if } |v| \ge R.
\end{align*}
Thus, from \eqref{Gsigma} we obtain that 
\begin{equation}   \label{Gthbis}
\begin{split}
&h_1(x) |u|^{\frac{1}{\theta_1}} - \sigma_3 \le G(x,u,0) \le \sigma_1 (1+|u|^{q_1})
\quad \mbox{ for a.e. } x\in\Omega, \mbox{ for all } u\in\R,\\
&h_2(x) |v|^{\frac{1}{\theta_2}}- \sigma_3 \le 
G(x,0,v) \le \sigma_1 (1+|v|^{q_2}) \quad \mbox{ for a.e. } x\in\Omega, \mbox{ for all } v\in\R
\end{split}
\end{equation}
for a positive constant $\sigma_3 >0$. Then, \eqref{si_pi} and \eqref{Gthbis} imply that
\[
p_1 s_1 <\ \frac{1}{\theta_1}\ \le q_1, \qquad
p_2 s_2 <\ \frac{1}{\theta_2}\ \le q_2,
\]
while from \eqref{thi<pi} it is
\[
p_1 <\ \frac{1}{\theta_1}, \qquad
p_2 <\ \frac{1}{\theta_2}.
\]
So, if condition \eqref{crit_exp} holds, 
without loss of generality we can take $q_1$, $q_2$ in $(g_1)$ large enough 
so that 
\begin{equation}    \label{esse1+1}
p_1(s_1+1) < q_1 < p_1^*(s_1+1)\quad\mbox{ and }\quad 
p_2(s_2+1) < q_2 < p_2^*(s_2+1).
\end{equation}
\end{remark}

In order to show that the $(wCPS)$ condition holds also in our supercritical setting, 
we need some preliminary results. 
\smallskip

Firstly, we note that,
taking $p > 1$ and $s \geq 0$, then straightforward computations give
\begin{equation}    \label{nablaconto}
\vert\nabla(\vert y\vert^{s} y) \vert^{p} = (s +1)^{p} \vert y\vert^{s p}\vert\nabla y\vert^{p} 
\quad \mbox{ a.e. in } \Omega, \ \mbox{ for all } y\in W^{1,p}_0(\Omega).
\end{equation}
Such an equality allows us to prove the following Rellich--type embedding theorem
(for the proof, see \cite[Lemma 3.8]{CPSSUP}).

\begin{lemma}    \label{lemma_s}
Taking $1 <p < N$ and $s >0$, let $(y_n)_n\subset W^{1,p}_0(\Omega) \cap L^\infty(\Omega)$ 
be a sequence such that
\[
\Bigg(\int_{\Omega}(1 +\vert y_n\vert^{sp}) \vert\nabla y_n\vert^{p}dx\Bigg)_n 
\quad \mbox{ is bounded.}
\]
Then, $y\in W^{1,p}_0(\Omega)$ exists such that $\vert y\vert^{s} y\in W^{1,p}_0(\Omega)$, 
too, and, up to subsequences, we have that
\[
\begin{split}
&y_n\ \rightharpoonup\ y  \quad\mbox{ weakly in } W^{1,p}_0(\Omega),\\  
&\vert y_n\vert^{s} y_n\ \rightharpoonup\ \vert y\vert^{s} y\quad
 \mbox{ weakly in } W^{1,p}_0(\Omega),\\
&y_n\ \rightarrow\ y \quad \mbox{ strongly in } L^{r}(\Omega) 
\ \mbox{ for each } r \in [1, p^{\ast}(s+1)[,\\
&y_n\ \rightarrow\ y \quad \mbox{ a.e. in } \Omega.
\end{split}
\]
\end{lemma}

Furthermore, we state the following boundedness result (for the proof, see \cite[Theorem II.5.1]{LU}).

\begin{lemma}    \label{Ladyz}
Let $\Omega$ be an open bounded subset of $\R^N$ and consider 
$y \in W_0^{1, p}(\Omega)$ with $p \le N$. 
Suppose that $\gamma > 0$ and $k_0\in\N$ exist such that 
\[
\int_{\Omega_{k}^{+}} \vert\nabla y\vert^p dx\ \leq\ 
\gamma\left( \int_{\Omega^{+}_{k}} (y - k)^r dx\right)^{\frac{p}{r}} 
+ \gamma \sum_{j=1}^m k^{\alpha_j} [\meas(\Omega^{+}_{k})]^{1-\frac{p}{N}+\varepsilon_j}\quad
\hbox{for all $k\geq k_0$,}
\]
with $\Omega^{+}_{k} =\lbrace x\in\Omega : y(x) > k\rbrace$ and 
$r$, $m$, $\alpha_j$, $\varepsilon_j$ positive constants such that
\[
1\leq r < p^{\ast}, \qquad \varepsilon_j > 0, \qquad p\leq\alpha_j < \varepsilon_j p^{\ast} + p.
\]
Then, $\displaystyle\esssup_{\Omega} y$ is bounded from above by a positive constant 
which can be chosen so that it depends only on 
$\meas(\Omega)$, $N$, $p$, $\gamma$, $k_0$, $r$, $m$, $\varepsilon_j$, 
$\alpha_j$, $\vert y\vert_{p^{\ast}}$ (eventually,
$\vert y\vert_{l}$ for some $l > r$ if $p^{\ast} = +\infty$).
\end{lemma}

As pointed out in Remark \ref{spiega}, the upper bounds in $(g_2)$ 
were not required so far, but will be essential in the incoming results. 
Therefore, some consequences of the estimates in \eqref{crit_exp} 
and \eqref{crit_expi} are needed.

\begin{remark} \label{rmkcrit}
Suppose $1 < p_1 < N$, $1 < p_2 < N$ and take $t_1$, $t_2$
as in \eqref{crit_expi}.
Following the ideas in Remark \ref{sucrit}, we can choose $t_3$ and $t_5$ in 
\eqref{crits12} so that
\begin{equation} \label{crits11}
1 < \frac{p_1 p_2^*(s_2+1)}{p_1 p_2^*(s_2+1) - N t_1} < t_3 < p^*_1(s_1+1),\quad
1 < \frac{p_2 p_1^*(s_1+1)}{p_2 p_1^*(s_1+1) - N t_2} < t_5 < p^*_2(s_2+1),
\end{equation}
as \eqref{crit_expi} implies that
\[
p_1 p_2^*(s_2+1) - N t_1 > 0\qquad \hbox{and}\qquad
1\ <\ \frac{p_1 p_2^*(s_2+1)}{p_1 p_2^*(s_2+1) - N t_1}\ <\ p^*_1(s_1+1),
\]
and also
\[
p_2 p_1^*(s_1+1) - N t_2 > 0\qquad \hbox{and}\qquad
1\ <\ \frac{p_2 p_1^*(s_1+1)}{p_2 p_1^*(s_1+1) - N t_2}\ <\ p^*_2(s_2+1).
\]
Then, $t_4$ and $t_6$ in \eqref{crits120} are such that
\begin{equation} \label{crits21}
t_1 \le t_4 < \ \frac{p_1}{N}\ p_2^*(s_2+1) \le p_2^*(s_2+1)
\quad \hbox{and}\quad
t_2 \le t_6 <\ \frac{p_2}{N} \ p_1^*(s_1+1) \le p_1^*(s_1+1).
\end{equation}
Clearly, \eqref{crits11} and \eqref{crits21} still hold if 
$p_1 =N$ and/or $p_2 =N$.
\end{remark}

\begin{remark}
In the set of hypotheses $(g_0)$--$(g_3)$,
by reasoning as in Remark \ref{sucrit} we can consider estimate \eqref{Gsigmamax}
with $\overline{q}_1$ and $\overline{q}_2$ as in \eqref{Ggrow}
but taking $t_j$, $j \in \{3,4,5,6\}$, as in Remark \ref{rmkcrit}.
Hence, from \eqref{crit_exp}, \eqref{esse1+1}, \eqref{crits11} and \eqref{crits21}
we infer that
\begin{equation}    \label{minmaxPerera1}
1< p_1 <\ \frac{\overline{q}_1}{s_1+1}\ < p_1^{\ast} \quad \mbox{ and } \
1 < p_2 <\ \frac{\overline{q}_2}{s_2+1}\ < p_2^{\ast}.
\end{equation}
\end{remark}

Finally, we are able to prove that the weak Cerami--Palais--Smale condition holds.

\begin{proposition}     \label{PropwCPS}
Under assumptions $(h_0)$--$(h_6)$ and $(g_0)$--$(g_3)$ 
functional $\J: X \to \R$, defined as in \eqref{functional}, 
satisfies $(wCPS)$ condition in $\R$.
\end{proposition}

\begin{proof}
Taking any $\beta\in\R$, let $((u_n, v_n))_n\subset X$ be a sequence such that
\begin{equation}  \label{3.15}
\J(u_n, v_n) \to \beta \quad \mbox{ and }\quad 
\Vert d\J(u_n, v_n)\Vert_{X^{\prime}} (1 + \Vert (u_n, v_n)\Vert_{X}) \to 0 
\quad \mbox{ as } n\to +\infty.
\end{equation}
We want to prove that a couple $(u, v)\in X$ exists such that 
\begin{enumerate}
\item[$(i)$] $(u_n, v_n) \to (u, v)$ in $W$ (up to subsequences),
\item[$(ii)$] $\J(u, v) = \beta$, $\  d\J(u, v) = 0$.
\end{enumerate}
To this aim, for simplicity, we organize our proof in the following steps:
\begin{enumerate}
\item both the sequences
\begin{equation}   \label{bounded}
\left(\int_{\Omega}(1 +\vert u_n\vert^{s_1p_1}) \vert\nabla u_n\vert^{p_1}dx\right)_n\;
\hbox{and}\;
\left(\int_{\Omega} (1 +\vert v_n\vert^{s_2p_2}) \vert v_n\vert^{p_2} dx\right)_n 
\quad \mbox{are bounded,}
\end{equation}
so, by applying Lemma \ref{lemma_s} a couple $(u, v)\in W$ exists such that 
also $(\vert u\vert^{s_1} u, \vert v\vert^{s_2} v)\in W$ and, up to subsequences,
we have: 
\begin{align}  
\label{4.7}
&(u_n, v_n)\ \rightharpoonup\ (u, v)  \quad\mbox{ weakly in } W,\\  
\label{4.8} 
&(\vert u_n\vert^{s_1} u_n, \vert v_n\vert^{s_2} v_n)\ \rightharpoonup\
(\vert u\vert^{s_1} u, \vert v\vert^{s_2} v)\quad \mbox{ weakly in } W,\\  
\label{4.10}
&(u_n, v_n)\ \rightarrow\ (u, v) \; \mbox{ in } L^{r_1}(\Omega)\times L^{r_2}(\Omega) 
\ \mbox{ if } 1 \le r_i < p_i^{\ast}(s_i +1),\, i \in \{1,2\},\\
 \label{4.9}
&(u_n, v_n)\ \rightarrow\ (u, v) \quad \mbox{ a.e. in } \Omega;
\end{align}
\item $(u, v)\in L^{\infty}(\Omega)\times L^{\infty}(\Omega)$;
\item for any $k>0$, define $T_k:\R\longrightarrow\R$ such that
\[ 
T_k t:= \begin{cases}
t &\hbox{ if } \vert t\vert\leq k\\
k \frac{t}{\vert t\vert} &\hbox{ if } \vert t\vert > k
\end{cases}
\]
and 
\[
\mathds{T}_k:(y_1,y_2)\in\R^2\mapsto\mathds{T}_k(y_1,y_2) = (T_ky_1, T_ky_2)\in\R^2,
\]
then, if $k\geq \max\lbrace\| (u, v)\|_{L}, R\rbrace + 1$ (with $R\geq 1$ as in our set of hypotheses), it is
\[     
\Vert d\J (\mathds{T}_k (u_n, v_n))\Vert_{X^{\prime}}\ \to\ 0
\quad \hbox{and}\quad
\J (\mathds{T}_k(u_n, v_n))\ \to\ \beta;
\]
\item $\|\mathds{T}_k (u_n, v_n) -(u, v) \|_{W}\ \to\ 0$ and then $(i)$ holds;
\item $(ii)$ is satisfied.
\end{enumerate}
For simplicity, here and in the following we will use 
the notation $(\varepsilon_n)_n$ for any infinitesimal sequence depending only on 
$((u_n, v_n))_n$.
Moreover, we denote by $c_i$ every positive constant which arises during our computations.\\
{\sl Step 1.} Firstly, we note that \eqref{dJzstar} and \eqref{3.15} imply
\begin{equation}   \label{epsin}
\frac{\partial\J}{\partial u}(u_n, v_n)[u_n] =\varepsilon_n \quad\mbox{ and }
\quad \frac{\partial\J}{\partial v}(u_n, v_n)[v_n] =\varepsilon_n.
\end{equation}
Thus, if we take $\theta_1$, $\theta_2$ as in $(h_5)$, $(g_3)$, 
and $s_1$, $s_2$ as in $(h_3)$, by reasoning as in 
\cite[{\sl Step 1} in Proposition 4.8]{CS}, 
from \eqref{functional}, \eqref{dJw}, \eqref{dJz}, \eqref{3.15}, \eqref{epsin}, 
assumptions $(h_1)$, $(h_3)$, $(h_5)$, $(g_3)$ together with estimate \eqref{Gsigmamax},
and direct computations it follows that
\[
\begin{split}
\beta +\varepsilon_n 
&= \J(u_n, v_n)-\theta_1\frac{\partial\J}{\partial u}(u_n, v_n)[u_n] -\theta_2\frac{\partial\J}{\partial v}(u_n, v_n)[v_n]\\
&\geq \mu_0 \mu_2\int_{\Omega}(1 +\vert u_n\vert^{s_1p_1})\vert\nabla u_n\vert^{p_1} dx 
+ \mu_0 \mu_2\int_{\Omega}(1 +\vert v_n\vert^{s_2p_2})\vert\nabla v_n\vert^{p_2} dx -c_1,
\end{split}
\]
which implies that \eqref{bounded} is satisfied and, up to subsequences, $(u,v) \in W$ 
exists such that \eqref{4.7}--\eqref{4.9} hold.\\
{\sl Step 2.} Due to the Sobolev Embedding Theorem, this step requires 
a proof only if either $p_1 \le N$ or $p_2 \le N$.
So, if $p_1 < N$ (when $p_1 = N$ the arguments can be simplified)
we want to prove that $u \in L^{\infty}(\Omega)$.
Arguing by contradiction, we assume that $u\notin L^{\infty}(\Omega)$
as either
\begin{equation}    \label{sup_u}
\esssup_{\Omega} u = +\infty
\end{equation}
or
\begin{equation}     \label{sup_menou}
\esssup_{\Omega}(-u) = +\infty. 
\end{equation}
If \eqref{sup_u} holds, then for any $k\in\N$ we have that
\begin{equation}         \label{mpos}
\meas(\Omega_{u,k}^{+}) > 0\qquad
\hbox{with}\quad 
\Omega_{u,k}^{+} = \{x\in\Omega :\ u(x) > k\},
\end{equation}
and for an integer $\tilde{k}>0$ we consider the function 
$R^{+}_{\tilde{k}}: t\in\R \mapsto R^{+}_{\tilde{k}} t\in\R$ defined as
\begin{equation}         \label{resto+}
R^{+}_{\tilde{k}} t = \begin{cases}
0 &\hbox{ if } t\leq \tilde{k}\\
t - \tilde{k} &\hbox{ if } t > \tilde{k}
\end{cases}.
\end{equation}
Now, we consider condition \eqref{mpos} for a
fixed integer $k > R$ (with $R\geq 1$ as in our setting of hypotheses) 
and, taking $\tilde{k} = k^{s_1 +1}$, for simplicity we put
\begin{equation}        \label{short}
w_n = |u_n|^{s_1}u_n, \qquad w = |u|^{s_1}u,
\end{equation}
and, as $|t|^{s_1} t >\tilde{k}\iff t>k$, 
we have that 
\begin{equation}        \label{shortu}
 \Omega_{u,k}^{+} = \{x\in\Omega :\ w(x) > \tilde{k}\}.
\end{equation}
Thus, from condition \eqref{4.8} and the sequentially weakly lower semicontinuity of $\|\cdot\|_{W_1}$
we have that
\[
\|R^{+}_{\tilde{k}}w\|_{W_1}\ \le\ \liminf_{n\to+\infty} \|R^{+}_{\tilde{k}}w_n\|_{W_1}, 
\]
i.e.,
\begin{equation}        \label{3.23}
\int_{\Omega_{u,k}^{+}} \vert\nabla w\vert^{p_1} dx
 \leq \liminf_{n\to +\infty}\int_{\Omega_{n,k}^{+}} \vert\nabla w_n\vert^{p_1} dx,
\end{equation}
with 
\[
\Omega_{n,k}^{+} = \lbrace x\in\Omega:\ u_n(x) > k\rbrace
= \{x\in\Omega :\ w_n(x) > \tilde{k}\}.
\]
On the other hand, by definition \eqref{resto+} with $\tilde{k}$ replaced with $k$,
it is $\Vert R^{+}_k u_n\Vert_{X_1} \leq \Vert u_n\Vert_{X_1}$,
 so from \eqref{dJzstar}, \eqref{3.15} and \eqref{mpos}, an integer $n_k\in\N$ exists such that
\begin{equation}      \label{m+mpos}
\frac{\partial\J}{\partial u}(u_n, v_n)\left[ R^{+}_k u_n\right]
 < \meas(\Omega_{u,k}^{+}) \quad \mbox{ for all } n\geq n_k.
\end{equation}
Then, by reasoning as in \cite[{\sl Step 2} in Proposition 4.8]{CS}, 
from \eqref{dJw}, hypotheses $(h_3)$, 
$(h_4)$ with $\mu_1 \leq 1$ (see Remark \ref{remmu1}), 
equality \eqref{nablaconto} and estimate \eqref{m+mpos} we obtain that
\begin{equation}    \label{mu1lambda}
\int_{\Omega^{+}_{n,k}} \vert\nabla w_n\vert^{p_1} dx 
\leq\frac{(s_1+1)^{p_1}}{\mu_0\mu_1} \meas(\Omega_{u,k}^{+}) + \int_{\Omega} G_u(x,u_n,v_n) R^{+}_k u_n dx
\quad \mbox{ for all } n\geq n_k.
\end{equation}
We claim that
\begin{equation}     \label{limGuRk}
\int_{\Omega} G_u(x,u_n,v_n) R^{+}_k u_n dx\ \to\ \int_{\Omega} G_u(x,u,v) R^{+}_k u\ dx.
\end{equation}
In fact, from \eqref{4.9} and $(g_0)$ we have that
\[
G_u(x, u_n, v_n)R^+_k u_n\to G_u(x, u, v)R^+_k u \quad\mbox{ a.e. in } \Omega,
\]
while, thanks to assumption $(g_2)$, formulae \eqref{gucr}, \eqref{crits12}, \eqref{crits11}, \eqref{crits21} and
\eqref{4.10} ensure the existence of $h\in L^1(\Omega)$ such that
\[
|G_u(x, u_n, v_n)R^+_k u_n|\leq\sigma(|u_n| + |u_n|^{q_1} +|u_n|^{t_3} +|v_n|^{t_4})\leq h(x) 
\quad\mbox{ for a.e. $x \in \Omega$,}
\]
so the Dominated Convergence Theorem implies \eqref{limGuRk}.\\
Thus, summing up, via \eqref{3.23}, \eqref{mu1lambda}, \eqref{limGuRk} and again \eqref{gucr}, 
from definition \eqref{resto+} ($\tilde{k}$ replaced with $k$) we infer that
\[
\begin{split}
\int_{\Omega_{u,k}^{+}} \vert\nabla w\vert^{p_1} dx&
\le c_2\left(\int_{\Omega} \vert R^+_k u\vert dx + 
\int_{\Omega}\vert u\vert^{q_1-1}|R^+_k u| dx 
+ \int_{\Omega}\vert v\vert^{t_1}\vert R^+_k u\vert dx +\meas(\Omega_{u,k}^{+})\right)\\
&\leq 
c_2\left(\int_{\Omega_{u,k}^{+}} \vert u\vert dx + 
\int_{\Omega_{u,k}^{+}}\vert u\vert^{q_1} dx 
+ \int_{\Omega_{u,k}^{+}}\vert u\vert \vert v\vert^{t_1} dx +\meas(\Omega_{u,k}^{+})\right),
\end{split}
\]
or better, from \eqref{crits12} but according to the choises in Remark \ref{rmkcrit},
by taking $\overline{q}_1 > 1$ as in \eqref{Ggrow} and being $u >1$ in $\Omega_{u,k}^{+}$, 
definition \eqref{short} implies that
\begin{equation}    \label{gradu31}
\int_{\Omega_{u,k}^{+}} \vert\nabla w\vert^{p_1} dx\leq 
c_3\left(\int_{\Omega_{u,k}^{+}}\vert w\vert^{\frac{\overline{q}_1}{s_1+1}} dx 
+ \int_{\Omega_{u,k}^{+}} \vert v\vert^{t_4} dx + \meas(\Omega_{u,k}^{+})\right).
\end{equation}
We claim that
\begin{equation}    \label{gradu311}
\int_{\Omega_{u,k}^{+}} \vert v\vert^{t_4} dx 
\le (\tau_{2,p^*_2} \| |v|^{s_2}v \|_{W_2})^{\frac{t_4}{s_2+1}}
[\meas(\Omega_{u,k}^{+})]^{1-\frac{t_4}{p_2^*(s_2+1)}}. 
\end{equation}
In fact, if $t_4= 0$ then \eqref{gradu311} reduces to
\[
\int_{\Omega_{u,k}^{+}} \vert v\vert^{t_4} dx = \meas(\Omega_{u,k}^{+}),
\]
while if $t_4 > 0$ from {\sl Step 1} we have that $z = |v|^{s_2}v \in W_2$, so 
\eqref{crits21} gives $\frac{p^*_2(s_2+1)}{t_4} >1$ and the
H\"older inequality with such an exponent, together with \eqref{Sobpi},
implies that
\[
\begin{split}
\int_{\Omega_{u,k}^{+}} \vert v\vert^{t_4} dx &= \int_{\Omega_{u,k}^{+}} \vert z\vert^{\frac{t_4}{s_2+1}} dx 
\le \vert z\vert_{p^*_2}^{\frac{t_4}{s_2+1}} [\meas(\Omega_{u,k}^{+})]^{1-\frac{t_4}{p_2^*(s_2+1)}} \\
& \le (\tau_{2,p^*_2} \| z \|_{W_2})^{\frac{t_4}{s_2+1}}
[\meas(\Omega_{u,k}^{+})]^{1-\frac{t_4}{p_2^*(s_2+1)}}. 
\end{split}
\]
On the other hand, if, for simplicity we put $r = \frac{\overline{q}_1}{s_1+1}$,
from \eqref{minmaxPerera1}, direct computations and, again, \eqref{Sobpi} we have that
\[
\begin{split}
\int_{\Omega_{u,k}^{+}} \vert w\vert^{\frac{\overline{q}_1}{s_1+1}} dx 
&\leq\ 2^{r -1}
\left(\int_{\Omega_{u,k}^{+}} \vert w - \tilde{k}\vert^{r} dx 
+ \tilde{k}^{r} \meas(\Omega_{u,k}^{+})\right)\\
&\leq\ 2^{r -1} \left((\tau_{1,r} \| w \|_{W_1})^{r-p_1}
\left(\int_{\Omega_{u,k}^{+}} \vert w - \tilde{k}\vert^{r} dx\right)^{\frac{p_1}{r}} 
+ \tilde{k}^{r} \meas(\Omega_{u,k}^{+})\right),
\end{split}
\]
which, together with \eqref{gradu311}, allows us to reduce \eqref{gradu31} to
the estimate
\begin{equation}    \label{gradu3}
\int_{\Omega_{u,k}^{+}} \vert\nabla w\vert^{p_1} dx\leq 
c_4\left(\left(\int_{\Omega_{u,k}^{+}} |w - \tilde{k}|^{r} dx\right)^{\frac{p_1}{r}} 
+ \tilde{k}^r \meas(\Omega_{u,k}^{+}) + [\meas(\Omega_{u,k}^{+})]^{1-\frac{t_4}{p_2^*(s_2+1)}}\right)
\end{equation}
with $c_4 = c_4(\|w\|_{W_1}, \|z\|_{W_2}) > 0$.\\
At last, as $p_1 < N$, from \eqref{crits21} we have that
\[
\begin{split}
&\meas(\Omega_{u,k}^{+}) = \meas(\Omega_{u,k}^{+})^{1-\frac{p_1}{N} + \eps_1}, 
\quad \hbox{with }\ \eps_1 = \frac{p_1}{N} > 0,\; \eps_1 p^*_1 + p_1 = p^*_1,\\
&\meas(\Omega_{u,k}^{+})^{1-\frac{t_4}{p_2^*(s_2+1)}} = 
\meas(\Omega_{u,k}^{+})^{1-\frac{p_1}{N} + \eps_2}, \quad \eps_2 = \frac{p_1}{N} - \frac{t_4}{p_2^*(s_2+1)} > 0,
\end{split}
\]
so, from \eqref{minmaxPerera1} and \eqref{shortu},
since \eqref{gradu3} holds for all $\tilde{k}$ large enough,
we have that Lemma \ref{Ladyz} applies and $\displaystyle \esssup_{\Omega} w < +\infty$
in contradiction to \eqref{sup_u}. \\
Similar arguments, but modified in a suitable way, ensures 
that even \eqref{sup_menou} cannot occur, then it has to be $u \in L^\infty(\Omega)$, 
and also that it has to be $v \in L^\infty(\Omega)$.\\
{\sl Step 3} The proof can be obtained by reasoning as in the proof of \cite[{\sl Step 3} in Proposition 4.8]{CS}
but with $m=2$ and by replacing the estimates in \cite[Remark 4.5]{CS} with those ones in Remark \ref{rmkcrit} 
together with \eqref{minmaxPerera1}, and also by using \eqref{4.10}
at the place of \cite[(4.19)]{CS}.\\
{\sl Steps 4 and 5.} The proofs are as in the corresponding steps 
of \cite[Proposition 4.8]{CS} (see also \cite[Proposition 4.6]{CP2}).
\end{proof}


\section{Existence and multiplicity results}    
\label{sec_main}

Now, we can state our leading results. To this aim, 
we refer to the decomposition of $X$ already introduced in \cite[Section 5]{CSS}. 
For the sake of convenience, 
here we recall the main issues. For $i\in\lbrace 1,2\rbrace$, the first eigenvalue of
$- \Delta_{p_i}$ in $W_i$ is given by
\begin{equation}\label{autoval_i}
\lambda_{i, 1}:=\inf_{y\in W_i\setminus\lbrace 0\rbrace}
\frac{\int_{\Omega}\vert\nabla y\vert^{p_i}dx}{\int_{\Omega}\vert y\vert^{p_i} dx}.
\end{equation}
Such an eigenvalue is simple, positive, isolated and has a unique eigenfunction $\varphi_{i,1}$
such that
\begin{equation}\label{eig}
\varphi_{i,1} > 0 \;\hbox{a.e. in $\Omega$,}\quad \varphi_{i,1} \in L^\infty(\Omega)
\quad\hbox{and}\quad |\varphi_{i,1}|_{p_i}=1
\end{equation}
(see, e.g., \cite{Lin}). Furthermore, a sequence of positive real numbers exists such that
\begin{equation}    \label{lambda_m}
0 <\lambda_{i, 1} <\lambda_{i, 2}\leq\dots\leq\lambda_{i, m}\leq\dots, 
\quad \mbox{ with } \lambda_{i, m}\nearrow +\infty \ \mbox{ as } m\to +\infty,
\end{equation}
with corresponding pseudo--eigenfunctions $(\psi_{i, m})_m$
which not only generate the whole space $W_i$, but are in $L^{\infty}(\Omega)$, too. 
Thus, $(\psi_{i, m})_m\subset X_i$, and, for any fixed $m\in\N$,
we consider
\[
V_{i, m} = {\rm span}\lbrace\psi_{i, 1},\dots,\psi_{i, m}\rbrace 
\]
and denote $Y_{i, m}$ its topological complement in $W_i$ so that
$W_i = V_{i,m}\oplus Y_{i,m}$ and the inequality
\begin{equation}\label{lambdan+1}
\lambda_{i,m+1}\ \int_\Omega|y|^{p_i} dx\ 
\leq\ \int_\Omega |\nabla y|^{p_i} dx \quad
\hbox{ for all }  y\in Y_{i,m}
\end{equation}
is satisfied (cf. \cite[Proposition 5.4]{CP2}).\\
Thus, for any $m\in\N$ definition \eqref{Wdefn1} implies that
\[
W = (V_{1, m}\times V_{2, m})\oplus(Y_{1, m}\times Y_{2, m})
\]
while from \eqref{Xdefn} it follows that
\[
X= (V_{1, m}\times V_{2, m})\oplus(Y_m^{X_1}\times Y_m^{X_2})
\]
where, for $i\in\lbrace 1, 2\rbrace$, it is 
$Y_{m}^{X_i} = Y_{i, m}\cap L^{\infty}(\Omega)\subset X_i$ and $X_i= V_{i, m}\oplus Y_m^{X_i}$, with
\[
\mbox{dim}(V_{i, m}) = m \quad \mbox{ and }\quad \mbox{codim}(Y_{m}^{X_i}) = m.
\]

Now, we are ready to provide our existence and multiplicity results.

\begin{theorem}    \label{ThExist}
Suppose that $A(x, t, \xi)$, $B(x, t, \xi)$ comply with assumptions 
$(h_0)$--$(h_6)$ and that a given function $G(x, u, v)$ satisfies hypotheses $(g_0)$--$(g_3)$. 
Furthermore, assume that a constant $\alpha_2>0$ exists such that
the following conditions hold:
\begin{enumerate}
\item[$(h_7)$] taking $p_1$, $p_2$ as in hypothesis $(h_1)$ and 
$s_1$, $s_2\geq 0$ as in assumption $(h_3)$, we have that
\[
\begin{split}
&A(x, t, \xi)\geq\alpha_2(1+|t|^{s_1p_1})\vert\xi\vert^{p_1}\quad\mbox{ a.e. in } \Omega, 
\mbox{ for all } (t, \xi)\in\R\times\R^N,\\
&B(x, t, \xi)\geq\alpha_2(1+|t|^{s_2p_2})\vert\xi\vert^{p_2}
\quad\mbox{ a.e. in } \Omega, \mbox{ for all } (t, \xi)\in\R\times\R^N;
\end{split}
\]
\item [$(g_4)$] taking $\lambda_{1,1}$ and $\lambda_{2,1}$ as in \eqref{autoval_i}, we have that
\[
\limsup_{(u,v)\to (0,0)} \frac{G(x, u, v)}{\vert u\vert^{p_1}+|v|^{p_2}}\ <\ 
\alpha_2\min\{\lambda_{1,1}, \lambda_{2,1}\}\quad
\hbox{uniformly a.e. in $\Omega$.}
\]
\end{enumerate}
Thus, functional $\J$ in \eqref{functional} possesses at least 
one nontrivial critical point in $X$; 
hence, problem \eqref{system} admits a nontrivial weak bounded solution.
\end{theorem}

\begin{theorem}     \label{ThMolt}
Suppose that $A(x,t,\xi)$, $B(x,t,\xi)$ and $G(x,u,v)$ satisfy 
hypotheses $(h_0)$--$(h_6)$, $(g_0)$--$(g_3)$. Moreover, if we assume also that:
\begin{enumerate}
\item[$(h_8)$] $A(x,\cdot,\cdot)$ and $B(x,\cdot,\cdot)$ are 
even in $\R\times\R^N$ for a.e. $x\in\Omega$;
\item[$(g_5)$] taking $\theta_1$, $\theta_2$ as in hypotheses $(h_5)$ and $(g_3)$,
we have that
\[
\liminf_{|(u, v)|\to +\infty}\frac{G(x, u, v)}
{\vert u\vert^{\frac{1}{\theta_1}} +\vert v\vert^{\frac{1}{\theta_2}}}\ >\ 0
\quad \hbox{uniformly a.e. in $\Omega$;}
\]
\item[$(g_6)$] $G(x,\cdot, \cdot)$ is even in $\R^2$ for a.e. $x\in\Omega$;
\end{enumerate}
then functional $\J$ in \eqref{functional} possesses an unbounded sequence 
of critical points $((u_m, v_m))_m\subset X$ such that $\J(u_m, v_m)\nearrow +\infty$; 
hence, problem \eqref{system} admits infinitely many distinct weak bounded solutions.
\end{theorem}

Finally, by reasoning as in \cite[Corollary 5.4]{CSS}, we can state this further 
multiplicity result since the supercritical growth 
in \eqref{minmaxPerera1} does not affect its proof.

\begin{corollary}
Let $p_1, p_2>1$ and suppose that the functions $A(x, t, \xi)$, $B(x, t, \xi)$ 
and $G(x, u, v)$ satisfy assumptions $(h_0)$--$(h_6)$, $(h_8)$, $(g_0)$--$(g_3)$ and $(g_6)$. Furthermore, if
\begin{enumerate}
\item[$(g_7)$] $\; \inf\{G(x,w,z):\ x \in \Omega,\ (w,z) \in \R^2 \ 
\hbox{such that $|(w,z)|=R$}\} > 0$, with $R$ as in $(g_2)$;  
\item[$(g_8)$] $\; \theta_1=\theta_2$, with $\theta_1$, $\theta_2$ as in $(h_5)$ and $(g_3)$;
\end{enumerate}
are satisfied too, the even functional $\J$ in \eqref{functional} possesses a sequence of critical 
points $((u_m,v_m))_m$ in $X$ such that $\J(u_m,v_m)\nearrow +\infty$; hence,
problem \eqref{system} admits infinitely many distinct weak bounded solutions.
\end{corollary}

Before turning to the proof of our main results, 
we observe that if assumption $(h_3)$, and then $(h_7)$, 
holds with $s_1 = s_2 =0$, then Theorem \ref{ThExist} reduces to \cite[Theorem 5.1]{CS}
while Theorem \ref{ThMolt} reduces to \cite[Theorem 5.2]{CS} but with $m=2$. 
Actually, the same holds true if both $p_1\ge N$ and $p_2\ge N$.
Thus, in order to improve such previous results, here we assume 
that either $s_1 >0$ or $s_2 >0$ and we define 
\begin{equation}  \label{ell12}
\ell_i(y) = \max\{\|y\|_{W_i}, \| |y|^{s_i} y\|_{W_i}\} \quad 
\hbox{if $ y\in X_i$,}\qquad \hbox{with $i \in \{1,2\}$,}
\end{equation}
and then
\begin{equation}    \label{ell}
\ell(u, v) = \max\{\|(u, v)\|_{W}, \| (|u|^{s_1} u, |v|^{s_2} v)\|_{W}\} 
\quad \hbox{if $ \ (u, v)\in X$.}
\end{equation}

From definitions \eqref{ell12} and \eqref{ell} we have that
\begin{equation}    \label{ell220}
[\ell_i(y)]^{p_i}\ \le\ \|y\|_{W_i}^{p_i} + \| |y|^{s_i} y\|_{W_i}^{p_i} \quad 
\hbox{if $ y\in X_i$,}\qquad \hbox{with $i \in \{1,2\}$,}
\end{equation}
and
\begin{equation}    \label{ell22}
\max\{\ell_1(u), \ell_2(v)\}\ \le\ \ell(u,v) \ \le \ \ell_1(u) + \ell_2(v)
\quad \hbox{for all $ \ (u, v)\in X$.}
\end{equation}
Moreover, taking $\bar{p}=\min\{p_1, p_2\}$, direct computations imply that
\begin{equation}  \label{ellnew}
[\ell_1(u)]^{p_1} + [\ell_2(v)]^{p_2} \ge \left[\frac{\ell(u, v)}{2}\right]^{\bar{p}} \quad
\mbox{ if } (u, v)\in X \mbox{ is such that } \ell(u, v)\ge 2.
\end{equation}

\begin{remark}\label{suell}
For both $i \in \{1,2\}$ definition \eqref{Xidefn} and identity \eqref{nablaconto}
imply that the function $y\mapsto\| |y|^{s_i} y\|_{W_i}$ is continuous and well defined in $(X_i, \|\cdot\|_{X_i})$
and so $\ell_i:X_i\to\R$ is continuous, too.
Thus, from \eqref{Xdefn} we have that $(u, v)\mapsto\| (|u|^{s_1} u, |v|^{s_2} v)\|_{W}$ 
is continuous and well defined in $(X, \|\cdot\|_X)$, then 
also $\ell:X\to\R$ is continuous with respect $\|\cdot\|_{X}$
and definition \eqref{ell} implies that
$\ell(u, v) \ge \|(u, v)\|_{W}$ for all $ \ (u, v)\in X$ with $\ell(0,0) = 0$.
\end{remark}

Throughout the remaining part of this section, for simplicity we assume that
\begin{equation}\label{ABzeros}
\int_{\Omega} A(x, 0, {\bf 0}_N) dx =0, \quad \int_{\Omega} B(x, 0, {\bf 0}_N) dx =0,
\end{equation}
with ${\bf 0}_N =(0,\dots, 0)\in\R^N$, and 
\begin{equation} \label{Gzeros}
\int_{\Omega} G(x, 0, 0) dx = 0.
\end{equation}
Differently, one can always replace $\J(u, v)$ in \eqref{functional}
with the new functional 
\[
\J^*(u, v) = \J(u, v) -\int_{\Omega} A(x, 0, {\bf 0}_N) dx - 
\int_{\Omega} B(x, 0, {\bf 0}_N) dx +\int_{\Omega} G(x,0,0) dx,
\]
since they share the same differential on $X$ and so the same critical points.

Moreover, we denote by $c_i$ every positive constant which arises during computations.
\smallskip

Now, we can prove our existence result.

\begin{proof}[Proof of Theorem \ref{ThExist}]
Firstly, hypothesis $(g_4)$ allows us to take $\bar{\lambda} > 0$ such that
\begin{equation}   \label{lambda_segn}
\limsup_{(u, v)\to (0,0)} \frac{G(x,u,v)}{\vert u\vert^{p_1}+|v|^{p_2}} 
<\bar{\lambda}< \alpha_2\min\{\lambda_{1,1}, \lambda_{2,1}\}
\quad \hbox{uniformly a.e. in $\Omega$.}
\end{equation}
Thus, from \eqref{lambda_segn} and direct computations,
estimate \eqref{Gsigmamax} ensures the existence of a constant $\sigma^*>0$ such that
\begin{equation}   \label{ls}
G(x,u,v) \leq \bar{\lambda} (|u|^{p_1} +|v|^{p_2}) +\sigma^*(|u|^{\bar{q}_1} +|v|^{\bar{q}_2})
\quad \hbox{for a.e. $x \in \Omega$, for all $(u,v) \in \R^2$,}
\end{equation}
with $\overline{q}_1, \overline{q}_2$ as in \eqref{Ggrow} so that \eqref{minmaxPerera1} holds.
Moreover, taking $s_1$, $s_2$ as in our setting of hypotheses and fixing any couple $(u,v) \in X$, 
from definition \eqref{functional}, condition $(h_7)$, estimate \eqref{ls}
together with \eqref{nablaconto} and \eqref{autoval_i} 
it follows that
\begin{equation}   \label{Jsum}
\begin{split}
\J(u, v)&\geq\left(\alpha_2-\frac{\bar{\lambda}}{\lambda_{1,1}}\right)\|u\|_{W_1}^{p_1} 
+\frac{\alpha_2}{(s_1+1)^{p_1}}\| |u|^{s_1} u\|_{W_1}^{p_1} -\sigma^* |u|_{\bar{q}_1}^{\bar{q}_1}\\
&\quad + \left(\alpha_2-\frac{\bar{\lambda}}{\lambda_{2,1}}\right)\|v\|_{W_2}^{p_2} 
+\frac{\alpha_2}{(s_2+1)^{p_2}}\| |v|^{s_2} v\|_{W_2}^{p_2} -\sigma^* |v|_{\bar{q}_2}^{\bar{q}_2},
\end{split}
\end{equation}
where from \eqref{minmaxPerera1} and the Sobolev inequality \eqref{Sobpi}
we have that
\begin{equation}     \label{uq1}
\int_{\Omega}|y|^{\bar{q}_i} dx = \int_{\Omega}||y|^{s_i} y|^{\frac{\bar{q}_i}{s_i+1}} dx
\leq c_1 \| |y|^{s_i} y\|_{W_i}^{\frac{\bar{q}_i}{s_i+1}}\quad \hbox{for all $y \in X_i$, with $i \in \{1,2\}$,}
\end{equation}
for a suitable $c_1 > 0$ independent of $i$.
Then, by using \eqref{uq1} in \eqref{Jsum}, from \eqref{lambda_segn} a positive constant $c_2 > 0$
exists such that definition \eqref{ell12},
estimate \eqref{ell220} and direct computations imply that
\[
\begin{split}
\J(u,v)&\geq c_2 (\|u\|_{W_1}^{p_1} + \| |u|^{s_1} u\|_{W_1}^{p_1}) - c_3 \| |u|^{s_1} u\|_{W_1}^{\frac{\bar{q}_1}{s_1+1}}
+ c_2 (\|v\|_{W_2}^{p_2} + \| |v|^{s_2} v\|_{W_2}^{p_2}) - c_3 \| |v|^{s_2} v\|_{W_2}^{\frac{\bar{q}_2}{s_2+1}}\\
&\ge [\ell_1(u)]^{p_1}\left(c_2 - c_3 [\ell_1(u)]^{\frac{\bar{q}_1}{s_1+1} - p_1}\right)
+ [\ell_2(v)]^{p_2} \left(c_2 - c_3 [\ell_2(v)]^{\frac{\bar{q}_2}{s_2+1} - p_2}\right),
\end{split}
\]
for a suitable $c_3>0$; hence, 
from \eqref{minmaxPerera1} and \eqref{ell22}, we obtain that 
\begin{equation}   \label{lconti}
\J(u,v)\ \geq\ [\ell_1(u)]^{p_1}\left(c_2 - c_3 [\ell(u,v)]^{\frac{\bar{q}_1}{s_1+1} - p_1}\right)
+ [\ell_2(v)]^{p_2} \left(c_2 - c_3 [\ell(u,v)]^{\frac{\bar{q}_2}{s_2+1} - p_2}\right).
\end{equation}
We note that, again from \eqref{minmaxPerera1}, a radius 
$r_0 > 0$ and a constant $\varrho_1$ can be found so that  
\[
c_2 - c_3 r_0^{\frac{\bar{q}_i}{s_i+1} - p_i}\ \ge\ \varrho_1 > 0\quad
\hbox{for both $i=1$ and $i=2$},
\]
thus, from \eqref{ell22} and \eqref{lconti} we infer that a constant $\varrho_0 > 0$ exists such that
\begin{equation}   \label{2i}
\ell(u, v) = r_0\quad \implies\quad \J(u, v)\geq\varrho_0.
\end{equation}
On the other hand, from $(h_0)$--$(h_2)$ and $(h_5)$ 
we have that \cite[Proposition 6.5]{CP2} implies the existence of
some constants $b^*_1$, $b_2^* >0$ such that 
\[
|A(x, t,\xi)|\leq b^*_1\left(1 + |t|^{\frac{1}{\theta_1}\left(1-\frac{\mu_2}{\eta_1}\right)}\right) 
+ b^*_2 \left(1 + |t|^{\frac{1}{\theta_1}\left(1-\frac{\mu_2}{\eta_1}\right) - p_1}\right) |\xi|^{p_1}
\]
a.e. in $\Omega$ and for all $(t,\xi) \in \R\times\R^N$,
with $\eta_1$, $\mu_2$ as in $(h_2)$, respectively $(h_5)$, and, without loss of generality,
we can assume 
$\frac{1}{\theta_1}\left(1-\frac{\mu_2}{\eta_1}\right) - p_1 >0$ 
(a priori, we can take either $\mu_2$ small enough or $\eta_1$ large enough).
Thus, taking $\varphi_{1,1}\in X_1$ as in \eqref{eig}, from \eqref{functional}, \eqref{Gthbis}, \eqref{ABzeros}
and direct computations, we obtain that
\[
\begin{split}
\J(\tau \varphi_{1,1}, 0) &\leq b_1^* \tau^{\frac{1}{\theta_1}(1-\frac{\mu_2}{\eta_1})}
\int_{\Omega}|\varphi_{1,1}|^{\frac{1}{\theta_1}(1-\frac{\mu_2}{\eta_1})} dx 
+ b_2^* \tau^{p_1}\int_{\Omega} |\nabla \varphi_{1,1}|^{p_1} dx\\
&+b_2^* \tau^{\frac{1}{\theta_1}(1-\frac{\mu_2}{\eta_1})}
\int_{\Omega}|\varphi_{1,1}|^{\frac{1}{\theta_1}(1-\frac{\mu_2}{\eta_1})-p_1} |\nabla\varphi_{1,1}|^{p_1} dx 
- \tau^{\frac{1}{\theta_1}}\int_{\Omega} h_1(x)|\varphi_{1,1}|^{\frac{1}{\theta_1}} dx + c_4,
\end{split}
\]
for a suitable $c_4 > 0$, which implies, from \eqref{thi<pi}, that 
\[
\J(\tau \varphi_{1,1}, 0)\to -\infty \quad\mbox{ as }\; \tau \to +\infty
\]
as \eqref{eig} and Remark \ref{rmkPerera0} ensure that $\int_{\Omega} h_1(x)|\varphi_{1,1}|^{\frac{1}{\theta_1}} dx>0$.\\
Hence, considering $r_0$, $\varrho_0$ so that \eqref{2i} holds,
a point $e_1\in X_1$ can be found so that 
\begin{equation} \label{3i}
\|(e_1, 0)\|_W >r_0 \quad\mbox{ and }\quad\J(e_1, 0)<\varrho_0.
\end{equation}
Finally, from \eqref{functional}, \eqref{ABzeros} and \eqref{Gzeros} it is $\J(0, 0)=0$, 
which, together with Remark \ref{suell}, \eqref{2i}, \eqref{3i} and Propositions \ref{smooth1} and \ref{PropwCPS},
ensures that Theorem \ref{mountainpass} applies and a critical point $(u, v)$ exists in $X$ such that $\J(u, v)\geq\varrho_0 >0$.
\end{proof}

In order to prove our multiplicity theorem, some geometric 
conditions are needed. 
In particular, if assumptions $(h_0)$--$(h_6)$ and $(g_0)$--$(g_3)$ hold, 
we are able to state the following results.

\begin{proposition}     \label{PropMolt}
For any fixed $\varrho\in\R$, an integer $m=m(\varrho)\geq 1$ 
and a radius $R_m >0$ exist such that
\[
(u, v)\in Y_m^{X_1}\times Y_m^{X_2}, \quad \ell(u, v) = R_m
\qquad \implies\qquad \J(u, v)\geq \varrho.
\]
\end{proposition}

\begin{proof}
Firstly, we note that \eqref{nablaconto} and \eqref{ell220}
imply that
\begin{equation}    \label{stim}
\int_\Omega(1+|y|^{s_i p_i}) |\nabla y|^{p_i} dx \ge\ 
\frac{1}{(s_i+1)^{p_i}}\ [\ell_i(y)]^{p_i}\quad
\hbox{if $y \in X_i$, for each $i \in \{1,2\}$.}
\end{equation}
Then, taking $(u,v)\in X$, from \eqref{functional}, 
\eqref{A13th}, \eqref{B24th}, \eqref{stim}, together with
\eqref{Gsigmamax} where $\bar{q}_1$, $\bar{q}_2$ satisfy \eqref{minmaxPerera1}, 
we obtain that
\begin{equation}    \label{Jgeq}
\begin{split}
\J(u,v)\ \geq\ &\frac{\mu_0(\mu_1\theta_1 +\mu_2)}{(s_1+1)^{p_1}}\ [\ell_1(u)]^{p_1} +
\frac{\mu_0(\mu_1\theta_2 +\mu_2)}{(s_2+1)^{p_2}} \ [\ell_2(v)]^{p_2}\\
&-\ \sigma_2 \int_{\Omega}|u|^{\bar{q}_1} dx
\ -\ \sigma_2 \int_{\Omega}|v|^{\bar{q}_2} dx - c_1,
\end{split}
\end{equation}
for some $c_1>0$. We note that for each $i \in \{1,2\}$
condition \eqref{minmaxPerera1} allows us to take $r_i >0$ so that
\[
\frac{r_i}{p_i} +\frac{\bar{q}_i -r_i}{p_i^{\ast}(s_i+1)} =1,
\]
then, reasoning as in \cite[Proposition 4.5]{CPSSUP}, 
from classical interpolation arguments, \eqref{Sobpi}
and \eqref{ell12} we obtain that
\[
\int_{\Omega}|y|^{\bar{q}_i} dx\ \leq\ c_2 [\ell_i(y)]^{\frac{\bar{q}_i -r_i}{s_i+1}} \
\left(\int_{\Omega}|y|^{p_i} dx\right)^{\frac{r_i}{p_i}}
\quad\mbox{ for all } y\in X_i,
\]
for a suitable constant $c_2 > 0$ independent of $i$.\\
Thus, fixing any $m\in\N$, from \eqref{lambdan+1} and, again,
\eqref{ell12} it follows that
\begin{equation}    \label{interp}
\int_{\Omega}|y|^{\bar{q}_i} dx\ \leq\ c_2 
 \lambda_{i, m+1}^{-\frac{r_i}{p_i}}[\ell_i(y)]^{\frac{r_is_i + \bar{q}_i}{s_i+1}}
\quad\mbox{ for all } y \in Y_m^{X_i},
\end{equation}
where from \eqref{minmaxPerera1} it is 
\begin{equation}    \label{interp4}
\frac{r_i s_i +\bar{q}_i}{s_i+1}\ >\ p_i.
\end{equation}
Hence, taking any couple $(u, v)\in Y_m^{X_1}\times Y_m^{X_2}$,
by using estimate \eqref{interp} in \eqref{Jgeq} we obtain that
\[
\J(u,v) \geq c_3\ [\ell_1(u)]^{p_1} 
- c_4  \lambda_{1,m+1}^{-\frac{r_1}{p_1}}[\ell_1(u)]^{\frac{r_1s_1 + \bar{q}_1}{s_1+1}}
+ c_3 [\ell_2(v)]^{p_2} 
- c_4  \lambda_{2, m+1}^{-\frac{r_2}{p_2}}[\ell_2(v)]^{\frac{r_2s_2 + \bar{q}_2}{s_2+1}} - c_1
\]
or better, from \eqref{ell22} and \eqref{interp4}, we have that
\begin{equation}    \label{Jgeq2}
\begin{split}
\J(u,v)\ \ge\ & [\ell_1(u)]^{p_1}\ \left(c_3 - c_4  \lambda_{1,m+1}^{-\frac{r_1}{p_1}}\
[\ell(u,v)]^{\frac{r_1s_1 + \bar{q}_1}{s_1+1}-p_1}\right)\\
&+ [\ell_2(v)]^{p_2} \left(c_3 - c_4  \lambda_{2, m+1}^{-\frac{r_2}{p_2}}\ 
[\ell(u,v)]^{\frac{r_2s_2 + \bar{q}_2}{s_2+1}-p_2}\right)\ -\ c_1.
\end{split}
\end{equation}
Now, for each $i \in \{1,2\}$, from \eqref{interp4} we can define $R_{i,m} > 0$ so that
\begin{equation}    \label{rayi}
c_4  \lambda_{i,m+1}^{-\frac{r_i}{p_i}}\ 
R_{i,m}^{\frac{r_is_i + \bar{q}_i}{s_i+1}-p_i}\ =\ \frac{c_3}{2}
\quad \iff\quad
R_{i,m}\ =\ \left(\frac{c_3}{2 c_4}\
 \lambda_{i,m+1}^{\frac{r_i}{p_i}}\right)^{\frac{s_i+1}{r_is_i + \bar{q}_i-p_i(s_i+1)}}
\end{equation}
and, since from \eqref{lambda_m} it follows that $R_{i,m} \nearrow +\infty$ as $m \to +\infty$,
we have that
\begin{equation}    \label{ray}
R_{m} := \min \{R_{1,m}, R_{2,m}\} \ \to\ +\infty\quad \hbox{as $m \to +\infty$}
\end{equation}
which implies $R_m \ge 2$ for all $m \ge m_0$ if $m_0\in \N$ is large enough.\\
So, for any $m \ge m_0$, taking $(u,v)\in Y_m^{X_1}\times Y_m^{X_2}$ 
such that $\ell(u,v) = R_m$,
from \eqref{ellnew} we have that
\begin{equation}    \label{ray2}
[\ell_1(u)]^{p_1} + [\ell_2(v)]^{p_2} \ \ge\ \left(\frac{R_m}{2}\right)^{\bar{p}},
\end{equation}
while from \eqref{Jgeq2}, by using \eqref{interp4}, \eqref{rayi} and the definition in \eqref{ray}, 
we obtain 
\[
\begin{split}
\J(u,v)\ge\ & [\ell_1(u)]^{p_1} \left(c_3 - c_4  \lambda_{1,m+1}^{-\frac{r_1}{p_1}}
R_m^{\frac{r_1s_1 + \bar{q}_1}{s_1+1}-p_1}\right)
+ [\ell_2(v)]^{p_2} \left(c_3 - c_4  \lambda_{2, m+1}^{-\frac{r_2}{p_2}} 
R_m^{\frac{r_2s_2 + \bar{q}_2}{s_2+1}-p_2}\right) - c_1\\
\ge\ & 
[\ell_1(u)]^{p_1} \left(c_3 - c_4  \lambda_{1,m+1}^{-\frac{r_1}{p_1}}
R_{1,m}^{\frac{r_1s_1 + \bar{q}_1}{s_1+1}-p_1}\right)
+ [\ell_2(v)]^{p_2} \left(c_3 - c_4  \lambda_{2, m+1}^{-\frac{r_2}{p_2}} 
R_{2,m}^{\frac{r_2s_2 + \bar{q}_2}{s_2+1}-p_2}\right) - c_1\\
=\ & 
\frac{c_3}{2}\ \left([\ell_1(u)]^{p_1} + [\ell_2(v)]^{p_2}\right) - c_1. 
\end{split}
\]
Thus, for any $m \ge m_0$ estimate \eqref{ray2} implies that
\begin{equation}    \label{Jgeq3}
\J(u,v)\ \ge\ \frac{c_3}{2} \ \left(\frac{R_m}{2}\right)^{\bar{p}} - c_1 
 \quad \hbox{if $(u,v)\in Y_m^{X_1}\times Y_m^{X_2}$ is  
such that $\ell(u,v) = R_m$.}
\end{equation}
Finally, we note that the proof follows from \eqref{ray} and \eqref{Jgeq3}.
\end{proof}

At last, by reasoning as in the first part of the proof of \cite[Theorem 5.2]{CS}
(we note that the computations do not involve the supercritical growth of $G(x,u,v)$
but only its lower bound coming from assumption $(g_5)$), 
the following result can be stated, too.

\begin{proposition}    \label{PropDim}
If also hypothesis $(g_5)$ holds, then
for any finite--dimensional subspace $V$ of $X$ a suitable radius $R_V>0$ exists such that
\[
\J(u, v)\leq 0 \quad\mbox{ for all  $(u, v)\in V$ such that } \ \|(u, v)\|_X\geq R_V.
\]
In particular, the functional $\J$ is bounded form above in $V$.
\end{proposition}

Now, we can prove our multiplicity results.

\begin{proof}[Proof of Theorem \ref{ThMolt}]
Firstly, we observe that \eqref{functional}, \eqref{ABzeros} and \eqref{Gzeros} 
give $\J(0, 0)=0$, while assumptions $(h_8)$ and $(g_6)$ imply that the functional $\J$ is even in $X$. 
Furthermore, taking any $r>0$, we set
\[
\M_r =\left\{ (u, v)\in X:\ \ell(u, v) =r\right\}. 
\]
By definition, $\M_r$ is the boundary of a symmetric neighborhood of the origin
which is bounded with respect to $\|\cdot\|_W$.\\
Now, fixing any $\varrho >0$, from Proposition \ref{PropMolt} 
it follows that an integer $m_{\varrho}\geq 1$ and a radius $r_\varrho=r_\varrho(m_{\varrho}) >0$ 
exist so that
\[
(u, v)\in \M_{r_\varrho} \cap (Y_{m_{\varrho}}^{X_1}\times Y_{m_{\varrho}}^{X_2}) 
\quad\implies\quad \J(u, v)\geq\varrho,
\]
while, by choosing $m>m_{\varrho}$, the $m$--dimensional space $V_m$ is such that
$\mbox{codim }Y_{m_{\varrho}} < \mbox{dim }V_m$, 
and from Proposition \ref{PropDim} a radius $R_{V_m} >0$ exists so that
\[
\J(u, v)\leq 0 \quad\mbox{ for all  $(u, v)\in V_m$ such that } \ \|(u, v)\|_X\geq R_{V_m}.
\]
Hence, assumption $(\cal{H}_{\varrho})$ in Theorem \ref{abstract} is verified. 
Then, the arbitrariness of $\varrho >0$ so that $(\cal{H}_{\varrho})$ holds,
together with Propositions \ref{smooth1} and \ref{PropwCPS}, allows us
to apply Corollary \ref{multiple} and the existence of a sequence of diverging critical levels 
for the functional $\J$ in $X$ is provided.
\end{proof}

\begin{proof}[Proof of Theorem \ref{ThModel}]
Taking $A(x,t,\xi)$ and $B(x,t,\xi)$ as in \eqref{ExAB}, from \eqref{exj0} it follows that conditions 
$(h_0)$--$(h_4)$ and $(h_6)$ hold. Moreover, if $G(x,u,v)$ is as in \eqref{ExG}, assumptions \eqref{exj0}--\eqref{exj02} and
Young inequality imply that $(g_0)$--$(g_2)$ are satisfied with
\[
t_1 = \gamma_2 \frac{q_1 - 1}{q_1 - \gamma_1},\quad
t_2 = \gamma_1 \frac{q_2 - 1}{q_2 - \gamma_2}.
\]
On the other hand, again from \eqref{exj0}, 
direct computations allow us 
to prove that hypotheses $(h_5)$ and $(g_3)$ are verified, too.
At last, also condition $(g_5)$ holds as \eqref{exj0} and
direct computations allow us to prove that for any $R \ge 2$ 
it is
\[
\frac{G(x,u,v)}{|u|^{\frac{1}{\theta_1}} + |v|^{\frac{1}{\theta_2}}} 
\ge \frac{1}{2} \min\left\{\frac{1}{q_1}, \frac{1}{q_2}\right\}
\quad \hbox{if $(u,v) \in \R^2$ is such that $|(u,v)| \ge R$.}
\] 
Then, since the symmetric assumptions $(h_8)$ and $(g_6)$ are trivially satisfied,
the thesis follows from Theorem \ref{ThMolt}. 
\end{proof}


\subsection*{Acknowledgments}
The authors wish to thank the Referee for her/his comments and suggestions 
which have been useful and have helped to improve this manuscript.



\begin{thebibliography}{99}

\bibitem{AMM}
W. Albalawi, C. Mercuri and V. Moroz, 
Groundstate asymptotics for a class of
singularly perturbed $p$--Laplacian problems in $\R^N$, 
\emph{Ann. Mat. Pura Appl. (4)} \textbf{199} (2020), 23-63.

\bibitem{AR} 
A. Ambrosetti and P.H. Rabinowitz, 
Dual variational methods in critical point theory and applications, 
{\em J. Funct. Anal.} {\bf 14} (1973), 349-381.

\bibitem{AS}
A. Ambrosetti and M. Struwe,  A note on the problem $-\Delta u =\lambda u +u|u|^{2^*-2}$, 
\emph{Manuscripta Math.} {\bf 54} (1986), 373–379.

\bibitem{ABO} D. Arcoya, L. Boccardo and L. Orsina, 
Critical points for functionals with quasilinear singular Euler--Lagrange 
equations, \emph{Calc. Var. Partial Differential Equations} {\bf 47} (2013), 159-180.

\bibitem{AG} G. Arioli and F. Gazzola, Existence and multiplicity results for quasilinear elliptic
differential systems, 
{\em Comm. Partial Differential Equations} {\bf 25} (2000), 125-153. 

\bibitem{BC}
A. Bahri and J.M. Coron, On a nonlinear elliptic equation involving the critical Sobolev exponent: the effect
of the topology of the domain, {\em Commun. Pure Appl. Math.} {\bf 41} (1988), 255–294.

\bibitem{BBF}
P. Bartolo, V. Benci and D. Fortunato,
Abstract critical point theorems and applications 
to some nonlinear problems with ``strong'' resonance at infinity,
\emph{Nonlinear Anal.} {\bf 7} (1983), 981-1012.

\bibitem{dF1} L. Boccardo and G. de Figueiredo,
Some remarks on a system of quasilinear elliptic equations,
{\em NoDEA Nonlinear Differential Equations Appl.} {\bf 9} (2002), 309–323.

\bibitem{BN}
H. Brezis and L. Nirenberg, 
Positive solutions of nonlinear elliptic equations involving critical Sobolev
exponents, 
\emph{Commun. Pure Appl. Math.} {\bf 36} (1983), 437–447.

\bibitem{CMPP}
A.M. Candela, E. Medeiros, G. Palmieri and K. Perera,
Weak solutions of quasilinear elliptic systems via the cohomological index,
\emph{Topol. Methods Nonlinear Anal.} {\bf 36} (2010), 1-18.

\bibitem{CP2}
A.M. Candela and G. Palmieri,
Infinitely many solutions of some nonlinear variational equations,
\emph{Calc. Var. Partial Differential Equations} {\bf 34} (2009), 495-530.

\bibitem{CP3} A.M. Candela and G. Palmieri, Some abstract critical point theorems
and applications. In: \emph{Dynamical Systems, Differential Equations and Applications} 
(X. Hou, X. Lu, A. Miranville, J. Su \& J. Zhu Eds), 
\emph{Discrete Contin. Dyn. Syst.} \textbf{Suppl. 2009} (2009), 133-142. 

\bibitem{CP2017}
A.M. Candela and G. Palmieri,
Multiplicity results for some nonlinear elliptic problems
with asymptotically $p$--linear terms,
\emph{Calc. Var. Partial Differential Equations} \textbf{56}:72 (2017).

\bibitem{CPSSUP}
A.M. Candela, G. Palmieri and A. Salvatore, 
Multiple solutions for some symmetric supercritical problems,
{\em  Commun. Contemp. Math.} {\bf 22} (2020), Article 1950075 (20 pages).

\bibitem{CS2020}
A.M. Candela and A. Salvatore,
Existence of radial bounded solutions for some quasilinear elliptic
equations in $\R^N$, \emph{Nonlinear Anal.} \textbf{191} (2020), Article 111625 (26 pp).

\bibitem{CSS}
A.M. Candela, A. Salvatore and C. Sportelli,
Existence and multiplicity results for a class of coupled quasilinear 
elliptic systems of gradient type, \emph{Adv. Nonlinear Stud.}.
DOI:10.1515/ans-2021-2121

\bibitem{CSS2}
A.M. Candela, A. Salvatore and C. Sportelli,
Bounded solutions for quasilinear modified Schrödinger equations, \emph{Calc. Var. Partial Differential Equations}, \textbf{61}, 220, (2022), 
https://doi.org/10.1007/s00526-022-02328-y

\bibitem{CS}
A.M. Candela and C. Sportelli,
Nontrivial solutions for a class of gradient--type 
quasilinear elliptic systems, \emph{Topol. Methods Nonlinear Anal.} T. 59, nr 2B,  (2022), 957–986, DOI: 10.12775/TMNA.2021.047.

\bibitem{CMP}
P. Candito, S.A. Marano and K. Perera,
On a class of critical $(p,q)$--Laplacian problems, 
\emph{NoDEA Nonlinear Differential Equations Appl.} {\bf 22} (2015), 1959–1972.

\bibitem{CFS}
G.  Cerami, D. Fortunato and M. Struwe,  Bifurcation and multiplicity results for 
nonlinear elliptic problems involving critical Sobolev exponents,  
\emph{Ann. I.H.P. Anal. Nonlineaire} {\bf 1} (1984), 341-350.

\bibitem{CLAPP}
M. Clapp and S. Tiwari, Multiple solutions to a pure 
supercritical problem for the $p$--Laplacian, \emph{Calc. Var. Partial Differential Equations} 
{\bf 55} (2016), 1-23.

\bibitem{COR}
J.M. Coron, Topologie et cas limite des injections de Sobolev, 
{\em C.R. Acad. Sci. Paris Sér. I Math.} {\bf 299} (1984), 209-212.

\bibitem{GV}
M. Guedda and L. Véron, Quasilinear elliptic equations involving critical Sobolev exponents,
{\em Nonlinear Anal. } {\bf 13} (1989), 879–902.

\bibitem{GPZ}
Z. Guo, K. Perera and W. Zou, On critical $p$--Laplacian systems, 
\emph{Adv. Nonlinear Stud.} {\bf 17} (2017), 641-659.

\bibitem{Jac}
S. Jacobs, \emph{An Isoperimetric Inequality for Functions Analytic in Multiply Connected Domains}, 
Mittag-Leffler Institute, Report {\bf 5}, 1972.

\bibitem{LU} O.A. Ladyzhenskaya and N.N. Ural'tseva, \emph{Linear and Quasilinear Elliptic
Equations}, Academic Press, New York, 1968.

\bibitem{Lieb}
E.H. Lieb, Sharp constants in the Hardy--Littlewood--Sobolev and related inequalities, 
{\em Ann. of Math. (2)} {\bf 118} (1983), 349–374.

\bibitem{Lin} P. Lindqvist, On the equation 
${\rm div} (|\nabla u|^{p-2}\nabla u) + \lambda |u|^{p-2}u =0$, {\em Proc. Amer. Math. Soc.}
{\bf 109} (1990), 157-164.

\bibitem{Lions}
P.L. Lions, 
The concentration--compactness principle in the calculus of variations. 
The locally compact case, II, 
{\em Ann. Inst. H. Poincaré Anal. Non Linéaire} {\bf 1} (1984), 223–283.

\bibitem{LWW}
J.Q. Liu, Y.Q. Wang and Z.Q. Wang, Soliton solutions for quasilinear Schr\"odinger equations, II,
{\em J. Differential Equations} \textbf{187} (2003), 473-493.

\bibitem{MP}
C. Mercuri and F. Pacella, On the pure critical exponent problem for the $p$--Laplacian, 
\emph{Calc. Var. Partial Differential Equations} {\bf 49} (2014), 1075-1090.

\bibitem{MSS}
C. Mercuri, B. Sciunzi and M. Squassina, 
On Coron's problem for the $p$--Laplacian, {\em J. Math. Anal. Appl.} {\bf 421} (2015), 362-369.

\bibitem{MeSq}
C. Mercuri and M. Squassina, 
Global compactness for a class of 
quasi--linear elliptic problems, {\em Manuscripta Math.} {\bf 140} (2013), 119-144.

\bibitem{MeWi}
C. Mercuri and M. Willem,  A global compactness result for the p-Laplacian involving critical nonlinearities,  
{\em Discret. Contin. Dyn. Syst.} {\bf 28} (2010), 469–493.

\bibitem{Per}
K. Perera, An abstract critical point theorem 
with applications to elliptic problems with combined nonlinearities, \emph{arXiv:2102.09131} (2021) (preprint).

\bibitem{Po}
S. Poho\v{z}aev, Eigenfunctions of the equation $\Delta u + \lambda f (u) = 0$, 
\emph{Soviet Math. Dokl.} \textbf{6} (1965), 1408–1411.

\bibitem{Str} 
M. Struwe, {\em Variational Methods. Applications to Nonlinear Partial 
Differential Equations and Hamiltonian Systems,} 4rd Edition,
Ergeb. Math. Grenzgeb. (4) {\bf 34}, Springer-Verlag, Berlin, 2008. 

\end{thebibliography}
\end{document}